\documentclass[reqno]{amsart}
\usepackage{amsfonts, mathtools, amsmath, amsthm, amssymb, latexsym, xfrac, mathrsfs, braket}
\usepackage{framed,graphicx,xcolor,enumitem,dsfont}
\usepackage[a4paper, left=2.5cm, right=2.5cm, top=2cm]{geometry}
\definecolor{shadecolor}{gray}{0.9}

\newcommand{\abs}[1]{\ensuremath{\left\vert#1\right\vert}}
\newcommand{\nrm}[1]{\left\lVert#1\right\rVert}
\DeclarePairedDelimiter\SET{\{}{\}}
\renewcommand{\set}{\SET}
\renewcommand{\Re}{\mathrm{Re}}
\renewcommand{\Im}{\mathrm{Im}}
\newcommand{\m}{\cdot}
\newcommand{\C}{\mathbb{C}}
\newcommand{\R}{\mathbb{R}}
\newcommand{\D}{\mathbb{D}}

\newcommand{\Z}{\mathbb{Z}}
\newcommand{\N}{\mathbb{N}}

\setlist[enumerate]{leftmargin=0.7cm, label=\alph*)}

\theoremstyle{plain}
\newtheorem{theorem}{Theorem}[section]
\newtheorem{corollary}[theorem]{Corollary}
\newtheorem{lemma}[theorem]{Lemma}
\newtheorem{proposition}[theorem]{Proposition}
\theoremstyle{definition}
\newtheorem{definition}[theorem]{Definition}
\theoremstyle{remark}
\newtheorem{remark}[theorem]{Remark}

\numberwithin{equation}{section}

\title[Boundedness of the Cherednik kernel and its limit transition from type BC to type A]
{Boundedness of the Cherednik kernel and its limit transition from type BC to type A}
\author{Dominik Brennecken} 
\address{Institut f\"ur Mathematik, Universit\"at Paderborn, Warburger Str. 100, D-33098 Paderborn, Germany}
\email{bdominik@math.upb.de}

\subjclass[2000]{Primary 33C67; Secondary 33C52}
\keywords{Root systems, special functions, Dunkl theory, functional analysis}

\begin{document}
\date{\today}

\begin{abstract}
We introduce a Cherednik kernel and a hypergeometric function for integral root systems and prove their relation to spherical functions associated with Riemannian symmetric spaces of reductive Lie groups. Furthermore, we characterize the spectral parameters for which the Cherednik kernel is a bounded function. In the case of a crystallographic root system, this characterization was proven by Narayanan, Pasquale and Pusti for the hypergeometric function. This result generalizes the Helgason-Johnson theorem from 1969, which characterizes the bounded spherical functions of a Riemannian symmetric space. The characterization for the Cherednik kernel is based on recurrence relations for the associated Cherednik operators under the dual affine Weyl group going back to Sahi. These recurrence relations are also used to prove a limit transition between the Cherednik kernel of type $A$ and of type $B$, which generalizes an already known result for the associated hypergeometric functions by Rösler, Koornwinder, and Voit.
\end{abstract}

\maketitle
\section*{Introduction}
In their original work on trigonometric Dunkl theory \cite{HO87, H87, O95, O00}, Heckman and Opdam introduced a hypergeometric function for crystallographic root systems with dependence on a continuous parameter $k$. Their theory generalizes radial analysis on Riemannian symmetric spaces of non-compact type $G/K$. Here $G$ is a non-compact semisimple and connected Lie group with finite center and a maximal compact subgroup $K$. Let $G=KAN$ and $\mathfrak{g}=\mathfrak{k}\oplus \mathfrak{a}\oplus \mathfrak{n}$ be associated Iwasawa decompositions on the Lie group and Lie algebra level with restricted roots $\Sigma\subseteq \mathfrak{a}$ and positive roots $\Sigma_+$. Consider the root system $R=2\Sigma$ and define $k_{2\alpha}$ as half of the dimension of the root space of $\mathfrak{g}$ associated with the root $\alpha$. For this specific $k$, the hypergeometric function $(F_k(R;\lambda,x))_{\lambda \in \mathfrak{a}_\C}$ of Heckman and Opdam associated with $(R,k)$ parameterizes the spherical functions $(\varphi_\lambda)_{\lambda \in \mathfrak{a}_\C}$ of $G/K$. More precisely, it holds that
\begin{equation}\label{SemisimpleHypGeoSpherical}
F_k(R;\lambda,X)=\varphi_\lambda(x)=\int_{K} e^{\braket{\lambda-\rho(k),H^G(xh)}}\;\mathrm{d} h, \quad X\in \mathfrak{a},
\end{equation}
where $\braket{\m,\m}$ is a certain inner product, $x$ is the image of $X$ under the exponential map of $G$, $H^G:G \to \mathfrak{a}\cong A$ is the Iwasawa projection and $\rho(k)=\frac{1}{2}\sum_{\alpha \in R_+} k_\alpha \alpha$. \\
We will begin the paper with a brief note on the fact that trigonometric Dunkl theory extends naturally to the case of integral root systems, i.e. the case where the root systems do not span $\mathfrak{a}$. This was already observed in \cite{BR23} for the root system $\mathrm{A}_{n-1}$. It poses a very natural extension and corresponds to passing from a semisimple Lie group to a reductive Lie group (of the Harish-Chandra class). It turns out that the corresponding extension of the hypergeometric function is in line with the extension from semisimple to reductive Lie groups, i.e. the relation \eqref{SemisimpleHypGeoSpherical} holds for reductive Lie groups $G$ and integral root systems $R$ with certain multiplicities $k$. \\
After this, we will continue with a generalization of the Helgason-Johnson theorem. In \cite{HJ69}, Helgason and Johnson characterized the parameters $\lambda$ for which the spherical function $\varphi_\lambda$ is bounded, namely
\begin{equation}\label{HelgasonJohnson}
\varphi_\lambda \text{ is bounded if and only if } \lambda \in C(\rho(k))+i\mathfrak{a},
\end{equation}
where $C(\rho(k))\subseteq \mathfrak{a}$ denotes the convex hull of the Weyl group orbit of $\rho(k)$. \\ 
This theorem was naturally extended in \cite{NPP14} to the Heckman-Opdam hypergeometric function by Narayanan, Pasquale, and Pusti for arbitrary non-negative parameter $k$. In this paper, we will prove the analogous result for the Cherednik kernel $G_k(R_+;\m,\m)$, which is the non-symmetric analog of the hypergeometric function with $F_k(R;\lambda,x)=\frac{1}{\#W}\sum_{w \in W}G_k(R_+;\lambda,wx)$ and Weyl group $W=W(R)$. Here it is easy to see that the boundedness of the Cherednik kernel implies the boundedness of the hypergeometric function and hence $\lambda \in C(\rho(k))+i\mathfrak{a}$. Boundedness of $F_k(R;\lambda,\m)$ for $\lambda \in C(\rho(k))+i\mathfrak{a}$ was proven in \cite{NPP14} by the inequality
\begin{equation}\label{BoundArgument}
\sup_{\lambda \in C(\rho(k))+i\mathfrak{a}} \abs{F_k(R;\lambda,x)} \le \sup_{\mu \in W.\rho(k)}F_k(R;\mu,x)=F_k(R;-\rho(k),x),
\end{equation}
where $F_k(R;-\rho(k),\m )\equiv 1$. The inequality \eqref{BoundArgument} is based on a maximum modulus principle and holds for the Cherednik kernel as well. However, since $G_k$ is not invariant under $W$, we do not know if the last equality in \eqref{BoundArgument} is true for the Cherednik kernel. This leads to the following approach in this paper, we prove recurrence relations for the Cherednik kernel in the parameter $\lambda$ under the dual affine Weyl group in the spirit of Sahi who established this for the non-symmetric Heckman-Opdam polynomials in \cite{S00a,S00b}. This will lead to new estimates between $G_k(R_+;\lambda,\m)$ and $G_k(R_+;w\lambda,\m)$ for $w \in W$ and in particular to the boundedness of $G_k(R_+;\lambda,\m)$ if $\lambda \in C(\rho(k))+i\mathfrak{a}$. \\
Finally, we will use the same recurrence relations of Sahi \cite{S00a,S00b} for the non-symmetric Heckman-Opdam polynomials to deduce a limit transition between the Cherednik kernels associated with the root systems
\begin{align*}
\mathrm{A}_{n-1}&=\set{\pm (e_j-e_i) \mid 1\le i<j\le n}, \\
\mathrm{BC}_n&=\set{\pm e_i,\pm 2e_i \mid 1\le i \le n} \cup \set{\pm (e_i\pm e_j) \mid 1\le i<j \le n},
\end{align*}
where $e_1,\ldots,e_n$ is the standard orthonormal basis of $\R^n$. More precisely, we will prove the following. If $k_1,k_2,k_3$ are the multiplicity values on $\pm e_i,\pm 2e_i,\pm (e_i\pm e_j)$, respectively, then
$$G_{(k_1,k_2,k_3)}^{\mathrm{BC}}(\lambda-\rho^{BC}(k_1,k_2,k_3),x) \xrightarrow[\substack{k_1+k_2 \to \infty \\ k_1/k_2 \to \infty}]{} G_{k_3}^{\mathrm{A}}(\lambda-\rho^{A}(k_3),\log(\cosh(x/2)^2)))$$
pointwise in $x \in \R^n$ and locally uniform in $\lambda\in\C^n$. This generalizes the same limit transition for the hypergeometric function proven by Rösler, Koornwinder, and Voit in \cite{RKV13}. We proceed as in \cite{RKV13}: first, we prove the limit for the Heckman-Opdam polynomials and then we extend the limit with Montel's and Carlson's theorem to the Cherednik kernels. While the second step can be verified in exactly the same way as in \cite{RKV13} for the hypergeometric function, the first step will be proven via the recurrence formulas of Sahi from \cite{S00a,S00b}. The approach via the eigenvalue equation for the symmetric Heckman-Opdam polynomials as in \cite{RKV13} does not seem to be an expedient solution in the non-symmetric case. The question about the non-symmetric analogue of the results in \cite{RKV13} for the Cherednik kernel proven here was initiated by Siddhartha Sahi in an oral communication with Margit Rösler. 

The organization of the paper is the following. In Section 1 we are introducing the notion of a Cherednik kernel and a hypergeometric function for integral root systems. These special functions will be connected with spherical functions of Riemannian symmetric spaces associated with reductive Lie groups in Section 2. The recurrence relations of Sahi for Cherednik operators and the consequences for the Cherednik kernel are contained in Section 3, while in Section 4 we will characterize the spectral parameters for which the Cherednik kernel is bounded. The limit transition between non-symmetric Heckman-Opdam polynomials for the root systems $\mathrm{A}_{n-1}$ and $\mathrm{BC}_n$ are proven in Section 5. Finally, in Section 6 we will extend the limit transition from the polynomial case to the Cherednik kernels.

\section*{Funding}
The author received funding from the Deutsche Forschungsgemeinschaft German Research Foundation (DFG), via RO 1264/4-1 and SFB-TRR 358/1 2023-491392403 (CRC Integral Structures in Geometry and Representation Theory).

\section{Trigonometric Dunkl theory for integral root systems}\label{IntegralRoots}
Let $R$ be an integral root system inside the Euclidean space $(\mathfrak{a},\braket{\cdot,\cdot})$ with norm $\abs{x}=\sqrt{\braket{x,x}}$, i.e. a finite subset with $0\notin R$ and
\begin{enumerate}[itemsep=5pt, topsep=5pt]
\item[(i)] $\braket{\alpha^\vee,\beta} \in \Z$ for all $\alpha,\beta \in R$ with the coroot $\alpha^\vee = \frac{2\alpha}{\braket{\alpha,\alpha}}$.
\item[(ii)] $s_\alpha R = R$ for all $\alpha \in R$ with the reflection $s_\alpha x = x-\braket{\alpha^\vee,x}$ in $\alpha^\perp \subseteq \mathfrak{a}$.
\end{enumerate} 
The integral root system $R$ is called crystallographic if $R$ spans $\mathfrak{a}$. 
The associated Weyl group will be denoted by $W=W(R)=\braket{s_\alpha,\alpha \in R}_{\text{group}}$.
We fix a system of positive roots $R_+ \subset R$ and a $W$-invariant function $k:R \to \C, \, \alpha \mapsto k_\alpha$ called a multiplicity function. The positive Weyl chamber is denoted by
$$\mathfrak{a}_+=\set{x \in \mathfrak{a} \mid \braket{\alpha,x}> 0 \text{ for all } \alpha \in R_+}.$$
Furthermore, we put
$$\mathfrak{a}=\mathfrak{s}\oplus \mathfrak{c} \quad \text{ with }\quad \mathfrak{s}\coloneqq \mathrm{span}_\R R \quad \text{ and } \quad \mathfrak{c}\coloneqq \mathfrak{s}^\perp.$$
Then $R$ is a crystallographic root system inside $\mathfrak{s}$ and $W$ acts trivially on $\mathfrak{c}$.
Let $R=\bigsqcup_{i=1}^m R^i$ be the decomposition of $R$ into non-trivial irreducible root systems $R^i$, $R_+^i = R_+ \cap R^i$ the subset of positive roots, $W_i=W(R_i)$ the associated Weyl groups such that $W=\prod_{i=1}^n W_i$ and $\mathfrak{s}=\oplus_{i=1}^m \mathfrak{s}_i$ the corresponding orthogonal decomposition of $\mathfrak{s}$ with $R^i\subseteq \mathfrak{s}_i$. We denote by $\pi_{\mathfrak{s}}$, $\pi_{\mathfrak{s}_i}$ and $\pi_{\mathfrak{c}}$ the orthogonal projections onto $\mathfrak{s},$ $\mathfrak{s}_i$ and $\mathfrak{c}$, respectively. We further extend $\pi_{\mathfrak{s}},$ $\pi_{\mathfrak{s}_i}$, and $\pi_{\mathfrak{c}}$ in a $\C$-bilinear way to the complexification $\mathfrak{a}_\C=\C\otimes\mathfrak{a}$.

\begin{definition}
Denote by $\partial_\xi$ the directional derivative in the direction $\xi \in \mathfrak{a}$.
As in the crystallographic case, we define the \textit{Cherednik operator} associated with $(R_+,k)$ in the direction $\xi \in \mathfrak{a}$ by
$$D_\xi(k)\coloneqq D_\xi(R_+,k)\coloneqq\partial_\xi-\braket{\rho(k),\xi}+\sum\limits_{\alpha \in R_+} k_\alpha \braket{\alpha,\xi}\frac{1-s_\alpha}{1-e^{-\alpha}}$$
with $e^{-\alpha}(x)=e^{-\braket{\alpha,x}}$ and $\rho(k)=\rho(R_+,k)=\frac{1}{2}\sum_{\alpha \in R_+} k_\alpha \alpha$. In particular, $D_\xi(k)$ equals $\partial_\xi$ for all $\xi \in \mathfrak{c}$.
\end{definition}

For $W$-invariant open subsets $\Omega \subseteq \mathfrak{a}$ we write $C^0(\Omega)=C(\Omega)$ for the space of complex-valued continuous functions on $\Omega$. For $n \in \N$ denote by $C^n(\Omega)$ the space of $n$-times continuously differentiable functions and write $C^\infty(\Omega)$ for the subspace of smooth functions. The subspaces of compactly supported functions are denoted by $C_c^n(\Omega)$ for $n \in \N_0\cup\set{\infty}$. All of these spaces are equipped with their natural locally convex topology.

\begin{remark}\label{ProdDecomp}
We make the following remarks:
\begin{enumerate}[itemsep=5pt, topsep=5pt]
\item[\rm (i)] The multiplicity functions $k:R \to \C$ on $R$ are in a one-to-one correspondence with sequences $(k_1,\ldots, k_m)$ of multiplicity functions $k_i:R^i \to \C$ on $R^i$ via the identification $k_i=k|_{R^i}$. 
\item[\rm (ii)] The Cherednik operators for integral root systems can be seen as Cherednik operators for crystallographic root systems in the following sense: Choose inside $\mathfrak{c}$ an arbitrary integral root system $R^0$ that spans $\mathfrak{c}$, positive roots $R^0_+ \subseteq R^0$ and consider the trivial multiplicity function $k_0=0$ on $R_0$. Then, for the root system $\widehat{R}=R^0 \sqcup R$ with the multiplicity function $\widehat{k}:\widehat{R} \to \C$ with $\widehat{k}|_{R^0}=k_0=0$ and $\widehat{k}|_{R}=k$, we obtain
$$D_\xi(R_+,k)=D_\xi(\widehat{R}_+,\widehat{k}),$$
where $\widehat{R} \subseteq \mathfrak{a}$ is now a crystallographic root system and
$$\rho(R_+,k)=\rho(\widehat{R}_+,\widehat{k})=\sum_{i=1}^n \rho(R_+^i,k_i)=\sum\limits_{i=0}^n \rho(R_+^i,k_i).$$
\item[\rm (iii)] From Cherednik operators for crystallographic root systems (cf. \cite{O95}), we have:
\begin{enumerate}[itemsep=5pt, topsep=5pt]
\item[\rm (1)] $D_\xi(R_+,k):C^n(\Omega) \to C^{n-1}(\Omega)$ is continuous for all open and $W$-invariant subset $\Omega \subseteq \mathfrak{a}$.
\item[\rm (2)] $\mathrm{supp}(D_\xi(R_+,k)f)\subseteq W.\mathrm{supp}\, f$ for all $f \in C^1(\Omega)$, where $\mathrm{supp}$ denotes the support.
\item[\rm (3)] $D_\xi(R_+,k)D_\eta(R_+,k)=D_\eta(R_+,k)D_\xi(R_+,k)$ as operators from $C^2(\Omega)$ to $C(\Omega)$.
\item[\rm (4)] If $g$ is $W$-invariant, then $D_\xi(k)(f\m g)=(D_\xi(k) f)\m g+ f\m (\partial_\xi g)$.
\end{enumerate}
\end{enumerate}
\end{remark}

\begin{proposition}\label{CherednikOperatorProjection}
For all $\xi \in \mathfrak{a}$ and $f \in C^1(\mathfrak{a})$ we have:
\begin{enumerate}[itemsep=5pt, topsep=5pt]
\item[\rm (i)] The Cherednik operator decomposes as
$$D_\xi(R_+,k)=  \partial_{\pi_{\mathfrak{c}}\xi} + D_{\pi_{\mathfrak{s}}\xi}(R_+,k) = \partial_{\pi_{\mathfrak{c}}\xi} + \sum\limits_{i=1}^n D_{\pi_{\mathfrak{s}_i}\xi}(R_+^i,k_i).$$
Furthermore, $D_{\pi_{\mathfrak{s}_i}\xi}(R_+,k)=D_{\pi_{\mathfrak{s}_i}\xi}(R_+^i,k_i)$.
\item[\rm (ii)] $D_\xi(R_+,k)(f\circ \pi_{\mathfrak{s}})=(D_{\pi_{\mathfrak{s}}\xi}(R_+,k) f) \circ \pi_{\mathfrak{s}}$.
\item[\rm (iii)] $D_\xi(R_+,k)(f\circ \pi_{\mathfrak{c}})=((\partial_{\pi_{\mathfrak{c}}\xi}-\braket{\rho(R_+,k),\xi}) f) \circ \pi_{\mathfrak{c}}$.
\item[\rm (iv)] $D_\xi(R_+,k)(f\circ \pi_{\mathfrak{s}_i})=((D_{\pi_{\mathfrak{s}_i}\xi}(R_+,k)-\braket{\rho(R_+,k)-\rho(R_+^i,k_i),\xi} )f) \circ \pi_{\mathfrak{s}_i}$.
\end{enumerate}
\end{proposition}

\begin{proof}
\
\begin{enumerate}[itemsep=5pt, topsep=5pt]
\item[\rm (i)] This is obvious from $\rho(R_+,k)=\sum_{i=1}^n \rho(R_+^i,k_i)$ and $\mathfrak{s}_i\perp \mathfrak{s}_j$ for $i\neq j$.
\item[\rm (ii)] Note that $\pi_{\mathfrak{s}}$ is a $W$-equivariant orthogonal projection, $R \subseteq \mathfrak{s}$ and $\rho(k) \in \mathfrak{s}$. So we have
\begin{align*}
D_\xi(R_+,k)(f\circ \pi_{\mathfrak{s}})(x)
& = (\partial_\xi - \braket{\rho(R_+,k),\xi})(f\circ \pi_{\mathfrak{s}})(x) + \sum\limits_{\alpha \in R_+} k_\alpha \braket{\alpha,\xi} \frac{f(\pi_{\mathfrak{s}}x)-f(\pi_{\mathfrak{s}}s_\alpha x)}{1-e^{-\braket{\alpha,x}}} \\
& = (\partial_{\pi_{\mathfrak{s}}(\xi)} - \braket{\rho(R_+,k),\pi_{\mathfrak{s}}\xi})f(\pi_{\mathfrak{s}}x) + \sum\limits_{\alpha \in R_+} k_\alpha \braket{\alpha,\pi_{\mathfrak{s}}\xi} \frac{f(\pi_{\mathfrak{s}}x)-f(s_\alpha\pi_{\mathfrak{s}}x)}{1-e^{-\braket{\alpha,\pi_{\mathfrak{s}}x}}} \\
& = D_{\pi_{\mathfrak{s}}\xi}(R_+,k)f(\pi_{\mathfrak{s}}x).
\end{align*}
\item[\rm (iii)] Since $W$ acts trivially on $\mathfrak{c}$ we have $\pi_{\mathfrak{c}}s_\alpha = s_\alpha\pi_{\mathfrak{c}}=\pi_\mathfrak{c}$. Hence, part (iii) is immediate from $\partial_\xi(f\circ \pi_{\mathfrak{c}})=\partial_{\pi_{\mathfrak{c}}\xi}f \circ \pi_{\mathfrak{c}}$, as the reflection part of the Cherednik operator vanishes.
\item[\rm (iv)] Since $W_i$ acts trivially on $\mathfrak{c}\oplus\bigoplus_{j\neq i}\mathfrak{s}_i$ we have $\pi_{\mathfrak{s}_i}s_\alpha = s_\alpha\pi_{\mathfrak{s}_i}=\pi_\mathfrak{s_i}$ for all $\alpha \in R_j$ with $j \neq i$. Therefore 
$$D_{\pi_{\mathfrak{s}_j}\xi}(R_+^j,k_j) (f \circ \pi_{\mathfrak{s}_i}) = -\braket{\rho(R_+^j,k_j),\xi}(f\circ \pi_{\mathfrak{s}_i})$$ for $j\neq i$, as well as $\partial_{\pi_{\mathfrak{c}}}(f\circ \pi_{\mathfrak{s}_i})=0$. Thus, $s_\alpha \pi_{\mathfrak{s}_i}=\pi_{\mathfrak{s}_i}s_\alpha$ for $\alpha \in R_i$ leads with part (i) to
\begin{align*}
D_\xi(R_+,k)(f\circ \pi_{\mathfrak{s}_i})(x) &= \left(D_{\pi_{\mathfrak{s}_i}\xi}(R_+^i,k_i)-\sum_{j\neq i}\braket{\rho(R_+^j,k_j),\xi}\right)(f\circ \pi_{\mathfrak{s}_i})(x) \\
&= ((D_{\pi_{\mathfrak{s}_i}\xi}(R_+,k)-\braket{\rho(R_+,k)-\rho(R_+^i,k_i),\xi} )f)(\pi_{\mathfrak{s}_i}x).
\end{align*}
\end{enumerate}
\end{proof}

Due to Remark \ref{ProdDecomp} (ii) we can understand the Cherednik operators of integral root systems as Cherednik operators of crystallographic root systems by extending (arbitrarily) the root system to a crystallographic one and putting the corresponding multiplicity to $0$. As the Cherednik operators does not depend on this extension, the work of Opdam in \cite[Theorem 3.15]{O95} shows that there exists to $(R_+,k)$, with $\Re \, k \ge 0$, and any $\lambda \in \mathfrak{a}_\C$ a unique analytic function (independent of the extension of the root system)
$$G_k^{\mathfrak{a}}(R_+;\lambda,\m):\mathfrak{a} \to \C,$$
called the Cherednik kernel associated with $(R_+,k)$ on $\mathfrak{a}$, with the property that $f(x)=G_k^{\mathfrak{a}}(R_+;\lambda,x)$ solves
$$\begin{cases}
D_\xi(R_+;k)f=\braket{\lambda,\xi}f, & \text{ for all } \xi \in\mathfrak{a}, \\
\hspace{32pt} f(0)=1.
\end{cases}$$
To be more precise, we have $G_k^{\mathfrak{a}}(R_+,\lambda,\m)=G_{\widehat{k}}^{\mathfrak{a}}(\widehat{R}_+,\lambda,\m)$ with $(\widehat{R}_+,\widehat{k})$ as in Remark \ref{ProdDecomp} (ii), where the latter Cherednik kernel is the one constructed by Opdam. \\
We define the hypergeometric function associated with $(R,k)$ on $\mathfrak{a}$ as
$$F_k^{\mathfrak{a}}(R,\lambda,x)=\frac{1}{\# W}\sum_{w\in W} G_k^{\mathfrak{a}}(R_+,\lambda,wx).$$
The hypergeometric function for crystallographic root systems was already introduced by Heckman, cf. \cite{HO87}. But, in contrast to the Cherednik kernel, we have in general $F_k^{\mathfrak{a}}(R;\lambda,x)\neq F_{\widehat{k}}^{\mathfrak{a}}(\widehat{R};\lambda,x)$, as already for $k=0$ we have
$$F_0^{\mathfrak{a}}(R;\lambda,x)=\frac{1}{\# W} \sum_{w \in W}e^{\braket{\lambda, wx}} \neq \frac{1}{\# W(\widehat{R})} \sum_{w \in W(\widehat{R})}e^{\braket{\lambda, wx}}=F_{0}^{\mathfrak{a}}(\widehat{R};\lambda,x).$$

\begin{theorem}\label{NonSpanningCherednik}
The Cherednik kernel $G_k^{\mathfrak{a}}(R_+;\m,\m)$ and hypergeometric function $F_k^{\mathfrak{a}}(R;\m,\m)$ decompose as
\begin{align*}
G_k^{\mathfrak{a}}(R_+;\lambda,x) &=e^{\braket{\pi_{\mathfrak{c}}\lambda,\pi_{\mathfrak{c}}x}}G_k^{\mathfrak{s}}(R_+;\pi_{\mathfrak{s}}\lambda,\pi_{\mathfrak{s}}x) \\
&= e^{\braket{\pi_{\mathfrak{c}}\lambda,\pi_{\mathfrak{c}}x}}G_{k_1}^{\mathfrak{s}_1}(R_+^1;\pi_{\mathfrak{s}_1}\lambda,\pi_{\mathfrak{s}_1}x) \cdots G_{k_m}^{\mathfrak{s}_m}(R^m_+;\pi_{\mathfrak{s}_m}\lambda,\pi_{\mathfrak{s}_m}x); \\
F_k^{\mathfrak{a}}(R;\lambda,x) &=e^{\braket{\pi_{\mathfrak{c}}\lambda,\pi_{\mathfrak{c}}x}}F_k^{\mathfrak{s}}(R;\pi_{\mathfrak{s}}\lambda,\pi_{\mathfrak{s}}x) \\
&= e^{\braket{\pi_{\mathfrak{c}}\lambda,\pi_{\mathfrak{c}}x}}F_{k_1}^{\mathfrak{s}_1}(R^1;\pi_{\mathfrak{s}_1}\lambda,\pi_{\mathfrak{s}_1}x) \cdots F_{k_m}^{\mathfrak{s}_m}(R^m;\pi_{\mathfrak{s}_m}\lambda,\pi_{\mathfrak{s}_m}x);
\end{align*}
for all $\lambda \in \mathfrak{a}_\C$ and $x \in \mathfrak{a}$. Moreover, $G_k^{\mathfrak{a}}$ and $F_k^{\mathfrak{a}}$ extend to holomorphic functions
$$\mathfrak{a}_\C \times (\mathfrak{a}+i\Omega) \to \C, \quad (\lambda,z) \to G_k(\lambda,z)$$
with the sector $\Omega=\set{x \in \mathfrak{a} \mid \abs{\braket{\alpha,x}}<\pi \text{ for all } \alpha \in R}$.
Moreover, the functions are holomorphic in $k$ on the set of multiplicity function with $\Re\, k_\alpha \ge 0$ for all $\alpha \in R$.
\end{theorem}

\begin{proof}
Consider for $\lambda \in \mathfrak{a}_\C$ the function
$$f(x)=e^{\braket{\pi_{\mathfrak{c}}\lambda,\pi_{\mathfrak{c}}z}}G_k^{\mathfrak{s}}(\pi_{\mathfrak{s}}\lambda,\pi_{\mathfrak{s}}x).$$
Then Proposition \ref{CherednikOperatorProjection} (i),(ii) and (iii) show that $f$ satisfies $D_\xi(R_+,k)f(x)=\braket{\lambda,\xi}f(x)$ and $f(0)=1$. By the uniqueness of the Cherednik kernel we have
$$f(x)=G_k^{\mathfrak{a}}(R_+;\lambda,x).$$
If we consider for $\lambda \in \mathfrak{a}_\C$ the function
$$f(x)=e^{\braket{\pi_{\mathfrak{c}}\lambda,\pi_{\mathfrak{c}}z}}G_{k_1}^{\mathfrak{s}_1}(\pi_{\mathfrak{s}_1}\lambda,\pi_{\mathfrak{s}_1}x) \cdots G_{k_m}^{\mathfrak{s}_m}(\pi_{\mathfrak{s}_m}\lambda,\pi_{\mathfrak{s}_m}x).$$
Then Proposition \ref{CherednikOperatorProjection} (i),(iii) and (iv) show that $f$ satisfies $D_\xi(R_+,k)f(x)=\braket{\lambda,\xi}f(x)$ and $f(0)=1$. The uniqueness of the Cherednik kernel leads to
$$f(x)=G_k^{\mathfrak{a}}(R_+;\lambda,x).$$
The statement for the hypergeometric functions follows from the one for the Cherednik kernel as $W=\prod_{i=1}^m W_i$ and $W_i$ acts trivial on $\mathfrak{c}\oplus\bigoplus_{j\neq i} \mathfrak{s}_j$. \\
The domain of holomorphicity was proven originally in \cite{KO08} for the hypergeometric function $F_k^{\mathfrak{s}}(\lambda,z)=\frac{1}{\# W}\sum_{w \in W}G_k^{\mathfrak{s}}(\lambda,wz)$. But an observation of the construction of $G_k^{\mathfrak{s}}$ in \cite[Theorem 3.15]{O95} shows that $G_k^{\mathfrak{s}}$ has the same domain of holomorphicity. Hence, it is true for $G_k^{\mathfrak{a}}$ and $F_k^{\mathfrak{a}}$.
\end{proof}

From now on, we will simply write
$$G_k=G_k^{\mathfrak{a}}\quad \text{ and }\quad  F_k=F_k^{\mathfrak{a}}.$$
The following proposition was already proven in \cite[Section 4]{BR23} in the crystallographic case, but an inspection of the proofs shows that they are still true for integral root systems.

\begin{proposition}\label{CherednikMinus}
Let $(D_\xi(R_+,k))_{\xi \in \mathfrak{a}}$ be the Cherednik operators associated with $(R_+,k)$. Put $w_0\coloneqq w_0(R_+)$ for the longest element of $W$ with respect to the choice of $R_+$. Then:
\begin{enumerate}[itemsep=5pt, topsep=5pt]
\item[\rm (i)] $wD_\xi(R_+,k)w^{-1}=D_{w\xi}(wR_+,k)$ for all $w \in W$.
\item[\rm (ii)] $D_\xi(R_+,k)f^- = -(D_\xi(-R_+,k)f)^-,$ where $f^-(x)=f(-x)$.
\item[\rm (iii)] $G_k(R_+,\lambda,z)=G_k(wR_+;w\lambda,wz)$ for all $w \in W$. \\
Moreover, the hypergeometric function $F_k$ does not depend on the choice of $R_+$.
\item[\rm (iv)] $G_k(R_+;\lambda,-z)=G_k(R_+;-w_0\lambda,w_0z)$.
\item[\rm (v)] $F_k(R;\lambda,-z)=F_k(R;-\lambda,z)$.
\end{enumerate}
\end{proposition}

In \cite{S08} the subsequent theorem was proven for crystallographic root systems, but it extends naturally to integral root systems by the decomposition from Theorem \ref{ProdDecomp}.

\begin{proposition}\label{GrowthCherednik}
For $\lambda,\mu \in \mathfrak{a}_\C$ and $x \in \mathfrak{a}$ the Cherednik kernel satisfies
\begin{enumerate}[itemsep=5pt, topsep=5pt]
\item[\rm (i)] $G_k>0$ on $\mathfrak{a}\times\mathfrak{a}$.
\item[\rm (ii)] $\abs{G_k(R_+;\lambda+\mu,x)} \le G_k(R_+;\Re\, \mu,x)\m e^{\max_{w \in W} \braket{\Re \, \lambda,wx}}.$
\item[\rm (iii)] $\abs{p(\tfrac{\partial}{\partial \lambda})q(\tfrac{\partial}{\partial x})G_k(R_+;\lambda,x)} \le C(1+\abs{x})^{\deg(p)}(1+\abs{\lambda})^{\deg(q)} F_k(R;0,x) e^{\max_{w \in W}\braket{w\Re(\lambda),x}}.$
\end{enumerate}
The estimates hold also for the hypergeometric function $F_k$.
\end{proposition}

We end this section with a small Theorem in rational Dunkl theory (the root systems does not have to satisfy the condition $\braket{\alpha^\vee,\beta}\in \Z$ for $\alpha,\beta \in R$), which is already know, but will be stated for a sake of completeness. Consider the rational Dunkl operators 
$$T_\xi(k)=T_\xi(R,k)=\partial_\xi+\sum\limits_{\alpha \in R_+}k_\alpha \braket{\alpha,\xi}\frac{1-s_\alpha}{\braket{\alpha,\m}}$$
as introduced by Dunkl in \cite{D89}, cf. also \cite{R03} for an overview on Dunkl theory.
For every $\lambda \in \mathfrak{a}_\C$ there exists a unique entire function $E_k(\lambda,\m)=E_k^{\mathfrak{a}}(R;\lambda,\m)$ with
$$E_k(\lambda,0)=1,\quad T_\xi(k)E_k(\lambda,\m)=\braket{\lambda,\xi}E_k(\lambda,\m), \quad \xi \in \mathfrak{a}.$$
The resulting entire function $E_k:\mathfrak{a}_\C \times \mathfrak{a}_\C \to \C$ is called the Dunkl kernel.
The Dunkl-Bessel function $J_k^{\mathfrak{a}}$ on $\mathfrak{a}$ associated with $(R,k)$ is then defined as
$$J_k(R;\lambda,z)=J_k(\lambda,z)=\frac{1}{\#W} \sum\limits_{w \in W} E_k(\lambda,wz), \quad \lambda,z \in \mathfrak{a}_\C.$$

With the same proof as in Proposition \ref{CherednikOperatorProjection} one has
\begin{proposition}\label{DunklOperatorProjection}
For all $\xi \in \mathfrak{a}$ and $f \in C^1(\mathfrak{a})$ we have:
\begin{enumerate}[itemsep=5pt, topsep=5pt]
\item[\rm (i)] The Dunkl operators decompose as
$$T_\xi(R,k)=  \partial_{\pi_{\mathfrak{c}}\xi} + T_{\pi_{\mathfrak{s}}\xi}(R,k) = \partial_{\pi_{\mathfrak{c}}\xi} + \sum\limits_{i=1}^n T_{\pi_{\mathfrak{s}_i}\xi}(R^i,k_i).$$
Furthermore, $T_{\pi_{\mathfrak{s}_i}\xi}(R,k)=T_{\pi_{\mathfrak{s}_i}\xi}(R^i,k_i)$.
\item[\rm (ii)] $T_\xi(R,k)(f\circ \pi_{\mathfrak{s}})=(T_{\pi_{\mathfrak{s}}\xi}(R,k) f) \circ \pi_{\mathfrak{s}}$.
\item[\rm (iii)] $T_\xi(R,k)(f\circ \pi_{\mathfrak{c}})=(\partial_{\pi_{\mathfrak{c}}\xi} f) \circ \pi_{\mathfrak{c}}$.
\item[\rm (iv)] $T_\xi(R,k)(f\circ \pi_{\mathfrak{s}_i})=(T_{\pi_{\mathfrak{s}_i}\xi}(R,k)f) \circ \pi_{\mathfrak{s}_i}$.
\end{enumerate}
\end{proposition}

Similar to Theorem \ref{NonSpanningCherednik} one can then proof the following:

\begin{theorem}\label{NonSpanningDunkl}
The Dunkl kernel $E_k^{\mathfrak{a}}(R;\m,\m)$ and the Bessel function $J_k^{\mathfrak{a}}(R;\m,\m)$ decompose as
\begin{align*}
E_k^{\mathfrak{a}}(R;\lambda,x) &=e^{\braket{\pi_{\mathfrak{c}}\lambda,\pi_{\mathfrak{c}}x}}E_k^{\mathfrak{s}}(R;\pi_{\mathfrak{s}}\lambda,\pi_{\mathfrak{s}}x) \\
&= e^{\braket{\pi_{\mathfrak{c}}\lambda,\pi_{\mathfrak{c}}x}}E_{k_1}^{\mathfrak{s}_1}(R^1;\pi_{\mathfrak{s}_1}\lambda,\pi_{\mathfrak{s}_1}x) \cdots E_{k_m}^{\mathfrak{s}_m}(R^m;\pi_{\mathfrak{s}_m}\lambda,\pi_{\mathfrak{s}_m}x); \\
J_k^{\mathfrak{a}}(R;\lambda,x) &=e^{\braket{\pi_{\mathfrak{c}}\lambda,\pi_{\mathfrak{c}}x}}J_k^{\mathfrak{s}}(R;\pi_{\mathfrak{s}}\lambda,\pi_{\mathfrak{s}}x) \\
&= e^{\braket{\pi_{\mathfrak{c}}\lambda,\pi_{\mathfrak{c}}x}}J_{k_1}^{\mathfrak{s}_1}(R^1;\pi_{\mathfrak{s}_1}\lambda,\pi_{\mathfrak{s}_1}x) \cdots J_{k_m}^{\mathfrak{s}_m}(R^m;\pi_{\mathfrak{s}_m}\lambda,\pi_{\mathfrak{s}_m}x); 
\end{align*}
for all $\lambda \in \mathfrak{a}_\C$ and $x \in \mathfrak{a}$.
\end{theorem}

\section{Connection to spherical functions of reductive Lie groups}
The aim of this section is to verify that the hypergeometric functions of integral root systems generalize spherical functions on a Riemannian symmetric space $G/K$ with $G$ of the Harish-Chandra class, and a maximal compact subgroup $K$. This connection is already known for semisimple $G$ and crystallographic root systems, cf. \cite{HS94, H97}. We further identify the spherical functions of the associated Riemannian symmetric space $G_0/K$ of Euclidean type, where $G_0$ is the Cartan motion group of $G$, as Dunkl type Bessel functions. For the background on Lie groups of the Harish-Chandra class, their structure theory, and their spherical functions we refer the reader to \cite{GV88}.

\begin{definition}
A (real) Lie group $G$ lies in the Harish-Chandra class (write $G \in \mathcal{H}$) if it satisfies the following four conditions:
\begin{enumerate}[itemsep=5pt, topsep=5pt]
\item[\rm (i)] $G$ has a reductive Lie algebra $\mathfrak{g}$, i.e. $\mathfrak{g}$ decomposes into a direct Lie algebra sum
$$\mathfrak{g}=\mathfrak{c} \oplus \mathfrak{s},$$
with $\mathfrak{c}$ abelian and $\mathfrak{s}$ semisimple. Hence, $\mathfrak{c}=\mathfrak{z}(\mathfrak{g})$ is the center and $\mathfrak{s}=[\mathfrak{g},\mathfrak{g}]$.
\item[\rm (ii)] $G$ has finitely many connected components.
\item[\rm (iii)] $\mathrm{Ad}(G)$ is contained in the connected complex adjoint group of $\mathfrak{g}_\C$, where $\mathrm{Ad}$ is the adjoint representation of $G$.
\item[\rm (iv)] $S\coloneqq \braket{\exp \mathfrak{s}}_{\text{group}}\subseteq G$ has finite center.
\end{enumerate}
\end{definition}

Let $G \in \mathcal{H}$ be a connected Lie group with exponential map $\exp$. By connectedness, $G$ decomposes into
$$G=C\m S \quad \text{ with } \quad C=\exp\mathfrak{c} \quad \text{ and } \quad S=\braket{\exp \mathfrak{s}}_{\text{group}}.$$
Let $K \subseteq G$ be a maximal compact subgroup which is connected due to the connectedness of $G$. From this, one has that $K \cap S \subseteq S$ is a (connected) maximal compact subgroup. In particular, if $\mathfrak{g}=\mathfrak{k}\oplus \mathfrak{p}$ is the associated Cartan decomposition of $\mathfrak{g}$, then $\mathfrak{s}=(\mathfrak{k}\cap \mathfrak{s}) \oplus (\mathfrak{p}\cap \mathfrak{s})$ is a Cartan decomposition of $\mathfrak{s}$. By \cite[Proposition 2.1.12]{GV88} there exists a non-degenerate bilinear form $B$ on $\mathfrak{g}_\C$ such that:
\begin{enumerate}[itemsep=5pt, topsep=5pt]
\item[\rm (i)] $B$ is negative definite on $\mathfrak{k}$ and positive definite on $\mathfrak{p}$.
\item[\rm (ii)] $\mathfrak{k} \perp^B \mathfrak{p}$.
\item[\rm (iii)] $B$ is invariant under $\mathrm{Ad}(G),$ $\mathrm{ad}\,\mathfrak{g}$ and the associated Cartan involution $\theta$.
\item[\rm (iv)] $B$ is real on $\mathfrak{g}\times\mathfrak{g}$.
\item[\rm (v)] $\mathfrak{c} \perp^B \mathfrak{s}$.
\end{enumerate}
Then 
\begin{equation}\label{HCInnerProduct}
\braket{X,Y}\coloneqq -B(X,\theta Y)
\end{equation}
defines an inner product on $\mathfrak{g}$ invariant under $\mathrm{Ad}(K)$, $\theta$ and with the same orthgonality properties as for $B$. The construction of $B$ is the following: on $\mathfrak{s}\times \mathfrak{s}$ it is given by the Cartan-Killing form and on $\mathfrak{c}$ one could choose any non-degenerate form which is negative definite on $\mathfrak{k}\cap \mathfrak{c}$ and positive definite on $\mathfrak{p}\cap \mathfrak{c}$. 

\begin{proposition}\label{RootsRelation}
Choose a maximal abelian subspace $\mathfrak{a} \subseteq \mathfrak{p}$ and the canonical maximal abelian subspace $(\mathfrak{a}\cap \mathfrak{s}) \subseteq (\mathfrak{p}\cap \mathfrak{s})$.
Let $\Sigma \subseteq \mathfrak{a}\cap \mathfrak{s}$ be the restricted roots of $\mathfrak{s}$ with respect to $\mathfrak{a}\cap \mathfrak{s}$ and let $\mathfrak{s}=\mathfrak{s}_0\oplus\bigoplus_{\alpha\in \Sigma} \mathfrak{s}_\alpha$ be the root space decomposition, i.e.
$$\mathfrak{s}_\alpha=\set{X \in \mathfrak{s} \mid [H,X]=\braket{\alpha,H}X \text{ for all } H \in \mathfrak{a}\cap \mathfrak{s}}.$$
Then we have:
\begin{enumerate}[itemsep=5pt, topsep=5pt]
\item[\rm (i)] The restricted roots of $\mathfrak{g}$ with respect to $\mathfrak{a}$ are precisely $\Sigma$.
\item[\rm (ii)] If $\mathfrak{g}=\mathfrak{g}_0\oplus \bigoplus_{\alpha \in \Sigma} \mathfrak{g}_\alpha$ is the root space decomposition of $\mathfrak{g}$ with respect to $\mathfrak{a}$, then 
$$\mathfrak{g}_\alpha=\begin{cases}
\mathfrak{s}_\alpha , & \alpha \in \Sigma, \\
\mathfrak{c}\oplus \mathfrak{s}_0, & \alpha = 0.
\end{cases}$$
\end{enumerate}
\end{proposition}

\begin{proof}
From $\mathfrak{g}=\mathfrak{c}\oplus \mathfrak{s}$ with center $\mathfrak{c}$ it is obvious that any root of $(\mathfrak{s},\mathfrak{a}\cap \mathfrak{s})$ is a root of $(\mathfrak{g},\mathfrak{a})$ and that
$$\mathfrak{g}_\alpha \supseteq \begin{cases}
\mathfrak{s}_\alpha , & \alpha \in \Sigma, \\
\mathfrak{c}\oplus \mathfrak{s}_0, & \alpha = 0.
\end{cases}$$
Let $\Sigma' \subseteq \mathfrak{a}$ be the roots of $(\mathfrak{g},\mathfrak{a})$, then from the root space decomposition
$$\mathfrak{g}_0 \oplus \bigoplus\limits_{\alpha \in \Sigma'} \mathfrak{g}_\alpha = \mathfrak{g}= \mathfrak{c}\oplus \mathfrak{s}= (\mathfrak{c}\oplus \mathfrak{s}_0) \oplus \bigoplus\limits_{\alpha \in \Sigma} \mathfrak{s}_\alpha$$
we obtain that $\Sigma=\Sigma'$ and $\mathfrak{g}_\alpha=\mathfrak{s}_\alpha$ for $\alpha \in \Sigma$.
\end{proof}
\vspace*{10pt}

Let $\mathfrak{g}=\mathfrak{k}\oplus \mathfrak{a}\oplus \mathfrak{n}$ and $G=KAN$ be Iwasawa decompositions. To be more precise, we choose a system of positive roots $\Sigma_+ \subseteq \Sigma$ and put $\mathfrak{n}=\bigoplus_{\alpha \in \Sigma_+} \mathfrak{g}_\alpha$, $A=\exp\mathfrak{a}$ and $N=\exp\mathfrak{n}$. From Proposition \ref{RootsRelation} we observe that in this case
$$\mathfrak{s}=(\mathfrak{k}\cap \mathfrak{s}) \oplus (\mathfrak{a}\cap \mathfrak{s}) \oplus \mathfrak{n} \quad \text{ and } \quad S=(K\cap S)(A \cap S) N$$
are Iwasawa decompositions for $\mathfrak{s}$ and $S$, respectively. We denote by $H^G:G \to \mathfrak{a}$ the Iwasawa projection defined in terms of the Iwasawa decomposition by $H^G(kan)=\log \, a$, where $\log$ is the inverse of the diffeomorphism $\exp:\mathfrak{a} \to A$. For $S$ we define in the same manner the Iwasawa projection $H^S:S \to (\mathfrak{a}\cap \mathfrak{s})$. Moreover, let $W$ be the Weyl group associated with the roots $\Sigma$. From $\mathfrak{s}\perp \mathfrak{c}$ and $\Sigma \subseteq \mathfrak{a} \cap \mathfrak{s}$, we obtain that $W$ acts trivially on $\mathfrak{a}\cap \mathfrak{c}$.\\
With the inner product \eqref{HCInnerProduct} the spherical functions of the Gelfand pair $(G,K)$ are given by
\begin{equation}\label{HarishChandraIntegral}
\varphi_\lambda^G(g)=\int_{K} e^{\braket{\lambda-\rho,H^G(gk)}}\;\mathrm{d} k, \quad \lambda \in \mathfrak{a}_\C,\; g \in G
\end{equation}
with $\rho=\frac{1}{2}\sum_{\alpha \in \Sigma_+}\dim(\mathfrak{g}_\alpha)\alpha$ and $\varphi_\lambda=\psi_\mu$ if and only if $\lambda \in W.\mu$, cf. \cite[Proposition 3.2.1]{GV88}. Moreover, $\varphi_\lambda$ is the unique $K$-biinvariant smooth function on $G$ that satisfies
$$D\varphi_\lambda=\gamma(D)(\lambda)\varphi_\lambda, \quad \varphi_\lambda(e)=1$$
for all $G$-invariant differential operators $D \in \D(G/K)$ and with the Harish-Chandra isomorphism 
$$\gamma:\D(G/K) \to \C[\mathfrak{a}]^W.$$ 
The spherical functions of the Gelfand pair $(S,K\cap S)$ are similarly given by
$$\varphi_\lambda^S(s)=\int_{K\cap S} e^{\braket{\lambda-\rho,H^S(sk)}}\;\mathrm{d} k, \quad \lambda \in (\mathfrak{a}\cap \mathfrak{s})_\C,\; s \in G.$$

\begin{lemma}\label{DecompositionSphericalFunction}
Let $\pi_\mathfrak{s}$ and $\pi_{\mathfrak{c}}$ be the orthogonal projections of $\mathfrak{a}$ onto $\mathfrak{s}$ and $\mathfrak{c}$, respectively. We extend the projections in a $\C$-bilinear way to the complexifications of the spaces. Then the spherical functions of $(G,K)$ and $(S,K \cap S)$, considered as $W$-invariant functions on $\mathfrak{a}\cong A$ and $\mathfrak{a}\cap \mathfrak{s} \cong A\cap S$ by restriction, respectively, are related by
$$\varphi_\lambda^G(x)= e^{\braket{\pi_{\mathfrak{c}}\lambda,\pi_{\mathfrak{c}}x}}\varphi^S_{\pi_{\mathfrak{s}}\lambda}(\pi_{\mathfrak{s}}x)$$
for all $\lambda \in \mathfrak{a}_\C$ and $x \in \mathfrak{a}$. 
\end{lemma}

\begin{proof}
For $x=x_c+x_s \in \mathfrak{a}$ with $x_c \in \mathfrak{a}\cap \mathfrak{c}$ and $x_s \in \mathfrak{a}\cap \mathfrak{s}$ we have
$$\exp x =\exp(x_c)\exp(x_s)=a_ca_s$$
with $a_c=\exp x_c \in A\cap C$ and $a_s=\exp x_s \in A \cap S$. The assertion holds by the following:
\begin{enumerate}[itemsep=5pt, topsep=5pt]
\item From $a_c \in A\cap C$ we obtain $H^G(ak)=\log a_c+H^G(a_sk)=H^G(a_sk)+x_c$. Thus, with $\rho \in \mathfrak{s} \perp \mathfrak{c}$ we have
$$\varphi^G_\lambda(x)=e^{\braket{\lambda-\rho,x_c}}\varphi_\lambda^G(x_s) = e^{\braket{\pi_{\mathfrak{c}}\lambda,x_c}}\varphi_\lambda^G(x_s)=e^{\braket{\pi_{\mathfrak{c}}\lambda,\pi_{\mathfrak{c}}x}}\varphi_\lambda^G(x_s).$$
\item Since $\mathfrak{k}=(\mathfrak{k}\cap \mathfrak{c})\oplus (\mathfrak{k}\cap \mathfrak{s})$, the connectedness of $K$ shows that $K=(K\cap C)\m (K\cap S).$ Hence, the quotient space $K/(K\cap S)$ can be identified with a subset of $K\cap C$ and by Weyl's integration formula there exists a unique probability measure $\mathrm{d} c$ on $K/(K\cap S)$ such that for all $a \in A \cap S$
\begin{equation}\label{WeylIntegration}
\int_K e^{\braket{\lambda-\rho,H^G(ak)}}\;\mathrm{d} k = \int_{K/(K\cap S)} \int_{K\cap S} e^{\braket{\lambda-\rho,H^G(acs)}} \;\mathrm{d} s \;\mathrm{d} c.
\end{equation}
But for $c \in K\cap C$, $a \in A\cap S$ and $s \in K\cap S$ we observe that $H^G(acs)=H^G(as)=H^S(as)$. Therefore, by equation \eqref{WeylIntegration} and $\rho \in \mathfrak{s}\perp \mathfrak{c}$ we have for $x=\log a \in \mathfrak{a}\cap \mathfrak{s}$
$$\varphi_\lambda^G(x)=\int_{K\cap S}e^{\braket{\lambda-\rho, H^S(as)}} \;\mathrm{d} s = \int_{K\cap S}e^{\braket{\pi_{\mathfrak{s}}\lambda-\rho, H^S(as)}} \;\mathrm{d} s=\varphi_{\pi_{\mathfrak{s}}\lambda}^S(x).$$
\end{enumerate}
\end{proof}
\vspace*{10pt}

The following theorem shows that the generalized hypergeometric functions for integral root systems, generalizes the spherical functions of Riemannian symmetric spaces related to Gelfand pairs of reductive Lie groups from the Harish-Chandra class.

\begin{theorem}\label{ReductiveSphericalHypergeometric}
The spherical functions of $(G,K)$ are related to the hypergeometric function $F_k(R;\m,\m)$ on $\mathfrak{a}$ associated with $R=2\Sigma \subseteq \mathfrak{a}$ and $k_{2\alpha}=\frac{\dim \mathfrak{g}_\alpha}{2}=\frac{\dim \mathfrak{s}_\alpha}{2}$ by the following formula for all $\lambda \in \mathfrak{a}_\C$, $x \in \mathfrak{a}$
$$\varphi_\lambda^G(x)=F_k(R;\lambda,x).$$
\end{theorem}

\begin{proof}
By the relation between the spherical functions of Riemannian symmetric spaces of non-compact type and the hypergeometric functions of crystallographic root systems \eqref{SemisimpleHypGeoSpherical} from \cite[Theorem 5.2.2]{HS94}, the assertion is an immediate consequence of Lemma \ref{DecompositionSphericalFunction} and Theorem \ref{NonSpanningCherednik}.
\end{proof}

Let $G_0\coloneqq K \ltimes \mathfrak{p}$ and $S_0\coloneqq (K\cap S)\ltimes (\mathfrak{p} \cap \mathfrak{s})$ be the Cartan motion groups associated with $(G,K)$ and $(S,K\cap S)$, where $K$ and $K\cap S$ act on $\mathfrak{p}$ and $(\mathfrak{p}\cap \mathfrak{s})$ by the adjoint representation, respectively.
We will identify the spherical functions of $(G_0,K)$ and $(S_0,K\cap S)$ as Dunkl type Bessel functions.

\begin{lemma}\label{SphericalFunctionsBessel} 
\
\begin{enumerate}[itemsep=5pt, topsep=5pt]
\item[\rm (i)] The spherical functions of $(S_0,K\cap S)$, considered as $W$-invariant functions on $\mathfrak{a}\cap \mathfrak{s}$, are
$$\psi_{\lambda}^{S_0}(x)=\int_{K\cap S} e^{\braket{\lambda,k x}} \;\mathrm{d} k, \quad \lambda \in (\mathfrak{a}\cap \mathfrak{s})_\C.$$
For every $(K\cap S)$-invariant differential operator $p(\partial)$ on $\mathfrak{p}\cap \mathfrak{s}$ with $p \in \C[\mathfrak{p}\cap\mathfrak{s}]^{K\cap S}$ we have $p(\partial)\psi^{S_0}_\lambda=p(\lambda)\psi_\lambda^{S_0}$.
\item[\rm (ii)] The spherical functions of $(G_0,K)$, considered as $W$-invariant function on $\mathfrak{a}$, are
$$\psi_{\lambda}^{G_0}(x)=\int_{K} e^{\braket{\lambda,k x}} \;\mathrm{d} k. \quad \lambda \in \mathfrak{a}_\C$$
satisfying $p(\partial)\psi_\lambda^{G_0}=p(\lambda)\psi_\lambda^{G_0}$ for all $p \in \C[\mathfrak{p}]^{K}$. Moreover, for $x \in \mathfrak{a}$ and $\lambda \in \mathfrak{a}_\C$ we have
$$\psi_\lambda^{G_0}(x)=e^{\braket{\pi_{\mathfrak{c}}\lambda,\pi_{\mathfrak{c}}x}}\psi_{\pi_{\mathfrak{s}}\lambda}^{S_0}(\pi_{\mathfrak{s}}x).$$
\end{enumerate}
\end{lemma}

\begin{proof}
\
\begin{enumerate}[itemsep=5pt, topsep=5pt]
\item This can be found in \cite[Proposition 4.8]{H84}.
\item The integral representation of the spherical functions has the same proof as in the case where $G$ is semisimple, cf. \cite[Theorem 3.2.3]{GV88}. Consider $x=x_c+x_s$ with $x_c \in \mathfrak{a}\cap \mathfrak{c}$ and $x_s\in \mathfrak{a}\cap \mathfrak{s}$. Then, as for all $k \in K$ we have $kx_c=x_c$ and $\mathfrak{s}\perp \mathfrak{c}$, we observe
$$\psi_\lambda^{G_0}(x)=e^{\braket{\lambda,x_c}}\psi_\lambda^{G_0}(x_s)=e^{\braket{\pi_{\mathfrak{c}}\lambda,\pi_{\mathfrak{c}}x}}\psi_\lambda^{G_0}(\pi_{\mathfrak{s}}x).$$
Finally, for $x \in \mathfrak{a}\cap \mathfrak{s}$ and $c \in K/(K\cap S)\subseteq K\cap C$ we have $cx=x$, as $C$ is contained in the kernel of the adjoint representation. Thus, as in the proof of Lemma \ref{DecompositionSphericalFunction} we obtain with $\mathfrak{s}\perp \mathfrak{c}$
$$\psi_\lambda^{G_0}(x)=\int_{K/(K\cap S)}\int_{K\cap S}e^{\braket{\lambda,kcx}} \;\mathrm{d} k \;\mathrm{d} c = \int_{K\cap S} e^{\braket{\pi_{\mathfrak{s}}\lambda,sx}} \;\mathrm{d} s=\psi_{\pi_{\mathfrak{s}}\lambda}^{S_0}(x).$$
\end{enumerate}
\end{proof}

\begin{theorem}
Let $(R,k)$ be as in Theorem \ref{ReductiveSphericalHypergeometric} and $J_k(R;\m,\m)$ the associated Bessel function on $\mathfrak{a}$. Then, the spherical functions of $(G_0,K)$ are given as $W$-invariant function on $\mathfrak{a}$ by
$$\psi_\lambda^{G_0}(x)=J_k(R;\lambda,x).$$
for all $\lambda\in\mathfrak{a}_\C$ and $x \in \mathfrak{a}$.
\end{theorem}

\begin{proof}
The theorem was proven by de Jeu in \cite{dJ06} in the case $\mathfrak{a}=\mathrm{span}_\R R$. The general case follows from the product decomposition of the Bessel function in Theorem \ref{NonSpanningDunkl} and Lemma \ref{SphericalFunctionsBessel} (ii).
\end{proof}

\begin{remark}
All the above results are still true if $G \in \mathcal{H}$ is not connected. As a maximal compact subgroup $K \subseteq G$ meets every connected component of $G$, cf. \cite[Proposition 2.1.7]{GV88}, the spherical functions of $(G,K)$ and $(G_e,K_e)$ ($G_e,K_e$ denoting the connected components of the unit $e$) can be identified.
\end{remark}

\section{Recurrence relations for the Cherednik operators}
In \cite{S00a,S00b} Sahi has obtained recurrence formulas for non-symmetric Heckman-Opdam polynomials via identities for the action of the dual affine Weyl group on Cherednik operators. These identities follow immediately from certain relations in a graded Hecke algebra, which has a representation in terms of Cherednik operators, see for instance \cite[last line p.79 and Corollary 2.9]{O95}. \\
For this section, assume that $R$ is a crystallographic irreducible root system inside the Euclidean space $(\mathfrak{a},\braket{\m,\m})$. Fix a system of positive roots $R_+\subseteq R$ with associated simple roots $\alpha_1,\ldots,\alpha_n$ and a multiplicity function $k:R \to \C, \, \alpha\mapsto k_\alpha$. Moreover, let $W$ be the Weyl group of $R$ and $s_i=s_{\alpha_i}$ the associated simple reflections. Consider the weight lattice
$$P\coloneqq \set{\lambda \in P \mid \braket{\lambda,\alpha^\vee} \in \Z \text{ for all } \alpha \in R}.$$
and the cone of dominant weights
$$P_+\coloneqq \set{\lambda \in P \mid \braket{\lambda,\alpha^\vee} \in \N_0 \text{ for all } \alpha \in R_+}.$$
Since $R$ is irreducible, there exists a unique highest short root denoted by $\beta$. We define the affine reflection
$$s_0:\mathfrak{a}\to \mathfrak{a}, \quad x\mapsto \beta+s_\beta x,$$
which is the affine reflection in the hyperplane $\set{x \in \mathfrak{a} \mid \braket{x,\beta^\vee}=1}$. The dual affine Weyl group can be expressed in terms of the root lattice $Q=\set{\sum_{i=1}^n m_i\alpha_i \mid m_i \in \Z}$ and in terms of $s_0,\ldots,s_n$ via
$$W^{\vee,\mathrm{aff}}=\braket{s_0,\ldots,s_n}_{\text{group}} \cong W\ltimes Q,$$
where $Q$ acts on $\mathfrak{a}$ via translations. The orbit space $W^{\vee,\mathrm{aff}}\backslash P$ has the following set of representatives
$$\mathcal{O}\coloneqq \set{\lambda \in P \mid \braket{\lambda,\alpha^\vee}\in\set{0,1} \text{ for all } \alpha \in R_+},$$
called the minuscle weights, cf. \cite{H90}. 

\begin{definition}
Consider on the dominant weights $P_+$ the partial order,
$$\mu \le \lambda \;\text{ iff } \;\lambda-\mu \in Q_+\coloneqq \sum_{i=1}^n \N_0 \alpha_i,$$
called dominance order. We extend the dominance order to $P$ via
$$\mu \trianglelefteq \lambda \; \text{ iff }\; \begin{cases}
\mu_+\le \lambda_+, &\text{if } \mu_+\neq \lambda_+, \\
\hspace{7pt}\lambda \le \mu,&\text{if } \mu_+=\lambda_+,
\end{cases}$$
where $\mu_+ \in W.\mu \cap P_+$ is the unique dominant weight in the $W$-orbit of $\mu$.
\end{definition}

We recall the notion of non-symmetric Heckman-Opdam polynomials (also called non-symmetric Jacobi-polynomials) $(E_\lambda(k;\m))_{\lambda \in P}$, introduced in \cite[Section 2]{O95}. For this, let $\mathcal{T}=\C(P)=\mathrm{span}_\C\set{e^\lambda \mid \lambda \in P}$ be the algebra of trigonometric polynomials associated with $P$, i.e. the group algebra of $P$.

\begin{definition}[\cite{O95,O00}]\label{DefHOPol}
The polynomial $E_\lambda(k;\m)=E_\lambda(R_+;k;\m)$ associated with $(R_+,k)$ is the unique element of $\mathcal{T}$ with
\begin{enumerate}[itemsep=5pt, topsep=5pt]
\item[\rm (i)] $E_\lambda(k;\m)=e^\lambda + \sum\nolimits_{\mu \triangleleft \lambda} c_{\mu\lambda}(k) e^\mu$.
\item[\rm (ii)] $D_\xi(k)E_\lambda(k;\m) = \braket{\widetilde{\lambda},\xi}E_\lambda(k;\m)$ for all $\xi \in \mathfrak{a}$ and some $\widetilde{\lambda} \in \mathfrak{a}$.
\end{enumerate}
In particular, $(E_\lambda(R_+;k;\m))_{\lambda \in P}$ form a basis of $\mathcal{T}$. Furthermore, the map $\lambda \to \widetilde{\lambda}$ is injective on $\mathfrak{a}$ and is given by
$$\widetilde{\lambda} = \lambda + \frac{1}{2}\sum\limits_{\alpha \in R_+} k_\alpha \epsilon(\braket{\alpha, \lambda})\alpha$$
with $\epsilon(t)=1$ for $t>0$ and $\epsilon(t)=-1$ for $t\le 0$.  \\
In particular, the polynomials are related to the Cherednik kernel via
\begin{equation*}\label{CherednikHOPoly}
G_k(\widetilde{\lambda},x)=\frac{E_\lambda(k;x)}{E_\lambda(k;0)}.
\end{equation*}
There is a similar basis for $\mathcal{T}^W$ of $W$-invariant trigonometric polynomials called symmetric Heckman-Opdam polynomials (or Jacobi polynomials) associated with $(R,k)$, denoted by $(P_\lambda(k;x))_{\lambda \in P_+}$ and uniquely characterized by
\begin{enumerate}[itemsep=5pt, topsep=5pt]
\item[\rm (i)] $P_\lambda(k;\m)=m_\lambda + \sum\nolimits_{\mu < \lambda} c_{\mu\lambda}(k) m_\mu$ with
$$m_\mu=\sum_{\eta \in W.\mu} e^\eta.$$
\item[\rm (ii)] $p(D(k))E_\lambda(k;\m) = p(\lambda+\rho(k))P_\lambda(k;\m)$ for all $W$-invariant polynomials $p \in \C[\mathfrak{a}]^W$.
\end{enumerate}
For $\lambda \in P_+$ one further has
\begin{equation}\label{SymNonSym}
P_\lambda(k;\m) = \frac{\# (W.\lambda)}{\# W}\sum\limits_{w \in W} E_{\lambda}(k;w\m).
\end{equation}
If $w_0 \in W$ is the longest element (with respect to $R_+$) then 
\begin{equation}\label{HypGeoHOPoly}
\frac{P_\lambda(k;x)}{P_\lambda(k;0)}=F_k(\lambda+\rho(k),x)=F_k(w_0\lambda-\rho(k),x)=\frac{1}{\#W} \!\sum\limits_{w \in W}G_k(w_0\lambda-\rho(k),wx)=\frac{1}{\#W}\! \sum\limits_{w \in W}\frac{E_\lambda(k;wx)}{E_\lambda(k;0)}. 
\end{equation}
\end{definition}

\begin{theorem}[\cite{S00a,S00b}]\label{Relation}
The dual affine Weyl group $W^{\vee,\mathrm{aff}}$ acts on $\mathcal{T}$ via its action on $P$, i.e. $w.e^{\lambda}=e^{w\lambda}$. Then, the Cherednik operators $(D_\xi(k))_{\xi \in \mathfrak{a}}$ associated with $(R_+,k)$ satisfy on $\mathcal{T}$:
\begin{enumerate}[itemsep=5pt, topsep=5pt]
\item[\rm (i)] For $j=1,\ldots,n$ we have
$$s_jD_\xi(k)-D_{s_j\xi}(k)s_j = -k_j\braket{\xi,\alpha_j}$$
with $k_j=(k_{\alpha_j}+2k_{2\alpha_j})$ and $k_{2\alpha_j}=0$ if $2\alpha_j \notin R$.
\item[\rm (ii)] The affine reflection $s_0$ satisfies for $k_0=k_\beta$
$$s_0(D_{s_\beta\xi}(k)+\braket{\xi,\beta})-D_\xi(k)s_0 = - k_0\braket{\xi,\beta}.$$
\end{enumerate}
\end{theorem}

\begin{corollary}[\cite{S00a,S00b}]\label{PolRec}
For $\lambda \in \mathfrak{a}$ we have:
\begin{enumerate}[itemsep=5pt, topsep=5pt]
\item[\rm (i)] $E_\lambda=e^\lambda$ if $\lambda \in \mathcal{O}$.
\item[\rm (ii)] If for $i=0,\ldots,n$ we have $s_i\lambda \neq \lambda$, then there exists a constant $c \in \R$ with 
$$c\m E_{s_i\lambda}(k;\m)= (s_i+c_i(k;\lambda))E_\lambda(k;\m)$$ 
with the constant
$$c_i(k;\lambda)=\begin{cases}
\tfrac{k_i}{\braket{\alpha_i^\vee,\widetilde{\lambda}}}, & \text{ if } i=1,\ldots, n\\
\tfrac{k_0}{1-\braket{\beta^\vee,\widetilde{\lambda}}}, & \text{ if } i=0. 
\end{cases}$$
Moreover, $c=1$ and $c_i(k;\lambda) \ge 0$ if $\braket{\alpha_i,\lambda}>0$ for $i=1,\ldots,n$ and $\braket{\beta^\vee,\lambda}<1$ for $i=0$.
\item[\rm (iii)] If $w=s_{i_1}\cdots s_{i_m} \in W^{\vee,\mathrm{aff}}$ is shortest element (in reduced expression) with $w\lambda=\overline{\lambda} \in \mathcal{O}$, then
$$E_\lambda(k;\m)=(s_{i_m}+c_m)\cdots (s_{i_1}+c_1)e^{\overline{\lambda}}$$
with $c_j=c_{i_j}(k;s_{i_{j-1}}\cdots s_{i_1}\overline{\lambda})$.
\item[\rm (iv)] The constants $c_{\mu\lambda}(k)$ in Definition \ref{DefHOPol} are non-negative rational functions in $k$.
\end{enumerate}
\end{corollary}

Note that the action of $W \subseteq W^{\vee,\mathrm{aff}}$ on $\mathcal{T}$ coincides with the usual action on functions $f:\mathfrak{a} \to \C$ via $(wf)(x)=f(w^{-1}x)$. But these actions do not coincide on the full dual affine Weyl group $W^{\vee,\mathrm{aff}}$. Indeed, for $f \in \mathcal{T}$ we have
$$(s_0f)(x)=e^{\braket{\beta,x}}f(s_\beta x) = (e^\beta \cdot (s_\beta f))(x).$$
Therefore, if $s_0$ acts on functions $f:\mathfrak{a}\to \C$ via $(s_0f)=e^\beta \m (s_\beta f)$, we obtain an action of $W^{\vee,\mathrm{aff}}$ on $C^1(\mathfrak{a})$. \\
By a density argument, we obtain that Theorem \ref{Relation} holds on $C^1(\mathfrak{a})$ with the described action of $W^{\vee,\mathrm{aff}}$ on $C^1(\mathfrak{a})$. From this we deduce the subsequent lemma which is the key tool to prove the Helgason-Johnson theorem for the Cherednik kernel in the next section.

\begin{lemma}\label{Recurrence}
Assume that $\Re\, k \ge 0$. The Cherednik kernel $G_k=G_k(R_+;\m,\m)$ associated with $(R_+,k)$ on $\mathfrak{a}$ satisfies:
\begin{enumerate}[itemsep=5pt, topsep=5pt]
\item[\rm (i)] For all $\lambda \in \mathfrak{a}_\C$ with $\braket{\lambda,\beta^\vee} \neq 1$, i.e. $s_0\lambda \neq \lambda$, we have
$$\left(1+\frac{k_\beta}{1-\braket{\lambda,\beta^\vee}}\right)G_k(\beta+s_\beta\lambda, \m ) = \left(e^{\beta} s_\beta+\frac{k_\beta}{1-\braket{\lambda,\beta^\vee}}\right) G_k(\lambda,\m).$$
\item[\rm (ii)] For all $j=1,\ldots,n$ and $\lambda \in \mathfrak{a}_\C$ with $\braket{\lambda,\alpha_j} \neq 0$, i.e. $s_j\lambda \neq \lambda$, we have
$$\left(1+\frac{k_j}{\braket{\lambda,\alpha_j^\vee}}\right)G_k(s_j\lambda, \m ) = \left(s_j+\frac{k_j}{\braket{\lambda,\alpha_j^\vee}}\right) G_k(\lambda,\m).$$
\end{enumerate}
Due to Theorem \ref{NonSpanningCherednik}, part (ii) is also true if the root system $R$ is integral (not necessarily crystallographic and irreducible). Part (i) remains true for every $\beta$ that is a highest short root of an irreducible component of the root system $R$.
\end{lemma}

\section{Helgason-Johnson theorem for the Cherednik kernel}
In \cite[Theorem 4.2]{NPP14}, the authors proved that the hypergeometric function $F_k(R;\lambda,\m)$ associated to a crystallographic root system $R$ is a bounded function on $\mathfrak{a}$ iff $\lambda \in C(\rho(k))+i\mathfrak{a}$, where 
$$C(\rho(k))=\set*{\left.\sum\nolimits_{w \in W} c_w \m w\rho(k) \right| c_w \ge 0,\, \sum\nolimits_{w \in W} c_w=1}$$
is the convex hull of the Weyl group orbit of $\rho(k)=\rho(R_+,k)=\frac{1}{2}\sum_{\alpha \in R_+}k_\alpha\alpha$. This generalizes the famous Helgason-Johnson theorem which characterizes the bounded spherical functions of a Riemannian symmetric space \cite{HJ69}. \\
For this section we assume that $R$ is an integral root system inside the Euclidean space $(\mathfrak{a},\braket{\m,\m})$. We fix a system of positive roots $R_+\subseteq R$ with associated simple roots $\alpha_1,\ldots,\alpha_n$ and a multiplicity function $k:R \to \C$ with $k \ge 0$. Moreover, let $W$ be the Weyl group of $R$ and $s_i=s_{\alpha_i}$ the associated simple reflections.

\begin{lemma}\label{HFprep}
There Cherednik kernel satisfies the following estimates:
\begin{enumerate}[itemsep=5pt, topsep=5pt]
\item[\rm (i)] For $\lambda \in \mathfrak{a}$ with $\braket{\alpha_j^\vee,\lambda}>0$ we have for all $x \in \mathfrak{a}$
\begin{align*}
G_k(\lambda,x) &\le \left(1+\frac{k_j}{\braket{\lambda,\alpha_j^\vee}}\right) \m G_k(s_j\lambda,s_j x), \\
G_k(s_j\lambda,x) &\le \left(1+\frac{k_j}{\braket{\lambda,\alpha_j^\vee}}\right) \m \max\set{G_k(\lambda,x),G_k(\lambda,s_jx)}.
\end{align*}
\item[\rm (ii)] For $\lambda \in \mathfrak{a}_+$ there is a constant $c_\lambda\ge 1$ with $c_{t\lambda} \underset{t\to\infty}{\longrightarrow} 1$ and for all $w \in W$
\begin{equation*}
\frac{1}{c_\lambda}\m \sup_{y \in W.x} G_k(\lambda,y) \le G_k(w\lambda,x) \le c_\lambda \m\sup_{y \in W.x} G_k(\lambda,y).
\end{equation*}
\item[\rm (iii)] For all $\lambda \in \mathfrak{a}$ there exist a constant $c_\lambda>0$ with $c_{t\lambda} \underset{t\to\infty}{\longrightarrow} 1$, such that for all $w \in W$
\begin{equation}\label{sandwich}
\frac{1}{c_\lambda}\m \sup_{y \in W.x} G_k(\lambda,x) \le \sup_{y \in W.x}G_k(w\lambda,y) \le c_\lambda \m\sup_{y \in W.x} G_k(\lambda,y).
\end{equation}
\end{enumerate}
\end{lemma}

\begin{proof}
\
\begin{enumerate}[itemsep=5pt, topsep=5pt]
\item[\rm (i)] Since $k\ge 0$, we have that $G_k>0$ on $\mathfrak{a}\times\mathfrak{a}$ and therefore Corollary \ref{Recurrence} leads to
$$G_k(\lambda,x) \le G_k(\lambda,s_j(s_j x)) + \frac{k_j}{\braket{\lambda,\alpha_j^\vee}} G_k(\lambda,s_j x) = \left(1+\frac{k_j}{\braket{\lambda,\alpha_j^\vee}}\right) \m G_k(s_j\lambda,s_jx).$$ 
Moreover, Corollary \ref{Recurrence} leads to
$$G_k(s_j\lambda,x) \le \left(1+\frac{k_j}{\braket{\lambda,\alpha_j^\vee}}\right)G_k(s_j\lambda,x) = G_k(\lambda,s_jx) + \frac{k_j}{\braket{\lambda,\alpha_j^\vee}} G_k(\lambda,x),$$
so that the second estimate holds.
\item[\rm (ii)] First of all, we can assume that $w\lambda \neq \lambda$, because otherwise we can pick $c_\lambda=1$. Pick a reduced expression $w=s_{i_1}\cdots s_{i_r} \in W$ and put $\lambda_{(j)}=s_{i_j}\cdots s_{i_r}\lambda$, so that $\lambda_{(j-1)}=s_{i_j}\lambda_{(j)}$ and $\braket{\lambda_{(j)},\alpha_{i_j}^\vee}>0$. Using induction with (i) the inequality \eqref{sandwich} holds with the constant
$$c_\lambda=\prod\limits_{j=1}^r\left(1+\frac{k_{i_j}}{\braket{\lambda_{(j)},\alpha_{i_j}^\vee}}\right),$$
satisfying obviously $c_{t\lambda} \underset{t\to\infty}{\longrightarrow} 1$ as $(t\lambda)_{(j)}=t\lambda_{(j)}$ for $t>0$.
\item[\rm (iii)]  With symmetry arguments of the inequality \eqref{sandwich}, it suffices to prove the statement for $\lambda \in \mathfrak{a}_+$. But this is a consequence of part (ii). 
\end{enumerate}
\end{proof}

\begin{theorem}\label{HJ}
Assume that $k>0$ and let $\lambda \in \mathfrak{a}_\C$. For the Cherednik kernel $G_k=G_k(R_+;\m,\m)$ associated with $(R_+,k)$ on $\mathfrak{a}$ the following statements are equivalent:
\begin{enumerate}[itemsep=5pt, topsep=5pt]
\item[\rm (i)] $G_k(\lambda,\m)$ is a bounded function on $\mathfrak{a}$.
\item[\rm (ii)] $\lambda \in C(\rho(k))+i\mathfrak{a}$.
\end{enumerate}
Moreover, the $G_k(\lambda,\m)$ with $\lambda \in C(\rho(k))+i\mathfrak{a}$ are uniformly bounded by a constant $C_k$.
\end{theorem}

\begin{proof}
Due to the product decomposition of $G_k$ from Theorem \ref{NonSpanningCherednik}, it suffices to prove the theorem for the case of a crystallographic root system.
{\rm (ii) $\Rightarrow$ (i)}: As in \cite[Proof of Theorem 4.2]{NPP14}, we can use $\abs{G_k(\lambda,x)}\le G_k(\Re\, \lambda,x)$ and the maximum modulus principle to obtain for fixed $x \in \mathfrak{a}$
$$\sup_{\lambda \in C(\rho(k))+i\mathfrak{a}} \abs{G_k(\lambda,x)} = \sup_{w \in W} G_k(w\rho(k),x).$$
Therefore, Lemma \ref{HFprep} and $G_k(-\rho(k),\m)\equiv 1$ lead to
$$\sup_{\lambda \in C(\rho(k))+i\mathfrak{a}} \abs{G_k(\lambda,x)}= \sup_{w \in W} G_k(w\rho(k),x) \le c_{-\rho(k)}\m \sup_{y \in W.x} G_k(-\rho(k),y)=c_{-\rho(k)} $$
with $C_k\coloneqq c_{-\rho(k)}>0$ independent of $x$ and $\lambda$.\\
{\rm (i) $\Rightarrow$ (ii)}: If $G_k(\lambda,\m)$ is bounded on $\mathfrak{a}$, then $F_k(\lambda,\m)$ is also bounded on $\mathfrak{a}$, so by \cite[Theorem 4.2]{NPP14} we have $\lambda \in C(\rho(k))+i\mathfrak{a}$.
\end{proof}

\begin{definition}\label{Cheredniktransform}
The Cherednik transform of a suitable function $f:\mathfrak{a} \to \C$ is defined by
$$\mathcal{H}_kf(\lambda)\coloneqq \int_{\mathfrak{a}} f(x)G_k(i\lambda,-x) \delta_k(x)\;\mathrm{d} x, \quad \lambda \in \mathfrak{a}$$
with the weight function
$$\delta_k(x)\coloneqq \prod\limits_{\alpha \in R_+} \abs{2\sinh \tfrac{\braket{\alpha,x}}{2}}^{2k_\alpha}.$$
Notice that due to Proposition \ref{CherednikMinus} we could replace in the definition of $\mathcal{H}_k$ the kernel $G_k(i\lambda,-x)$ by $G_k(-iw_0\lambda,w_0x)$, where $w_0$ is the longest element of $W$ with respect to $R_+$. This transform was studied deeply in \cite{O95,S08}. The inverse Cherednik transform is defined by
$$\mathcal{I}_kf(x)\coloneqq \int_{\mathfrak{a}} f(\lambda)G_k(i\lambda,x) \nu(i\lambda) \;\mathrm{d} \lambda,$$
with the weight function
$$\nu(\lambda)=c\m \prod\limits_{\alpha \in R_+} \frac{\Gamma(\braket{\lambda,\alpha^\vee}+k_\alpha+\frac{1}{2}k_{\alpha/2})\Gamma(-\braket{\lambda,\alpha^\vee}+k_\alpha+\frac{1}{2}k_{\alpha/2}+1)}{\Gamma(\braket{\lambda,\alpha^\vee}+\frac{1}{2}k_{\alpha/2})\Gamma(-\braket{\lambda,\alpha^\vee}+\frac{1}{2}k_{\alpha/2}+1)},$$
where $c$ is a suitable normalization constant and $k_{\alpha/2}=0$ for $\alpha/2 \notin R$.
\end{definition}

\begin{remark}
Recall Theorem \ref{ReductiveSphericalHypergeometric}, i.e. the relation between the hypergeometric function and spherical functions of a Riemannian symmetric space $G/K$ associated with a Lie group $G \in \mathcal{H}$ of the Harish-Chandra class. Consider $f \in C_c^\infty(\mathfrak{a})^W$ and the associated unique $K$-biinvariant function $F \in C_c^\infty(G)$ with $F(\exp(x))=f(x)$ for all $x \in \mathfrak{a}$. Then, the Cherednik transform of $f$ can be rewritten
\begin{equation}\label{CherednikSphericalTransform}
\mathcal{H}_kf(\lambda) = \int_{\mathfrak{a}} f(x) F_k(i\lambda,-x) \delta_k(x)\;\mathrm{d} x 
= \int_{\mathfrak{a}} f(x)\varphi_{i\lambda}^G(e^x) \delta_k(x)\;\mathrm{d} x 
= \int_G F(g) \varphi_{i\lambda}^G(g) \;\mathrm{d} g,
\end{equation}
where the last equality can be found in \cite[Proposition 2.4.6]{GV88}. The last integral in \eqref{CherednikSphericalTransform} is the Harish-Chandra transform on $G/K$, i.e. the spherical Fourier transform associated with $(G,K)$.
\end{remark}

We are now in the position to prove the non-symmetric generalization of \cite[Corollary 5.1]{NPP14}, which is a generalized Riemann-Lebesgue lemma. Consider the usual Lebesgue space $L^1(\mathfrak{a},\delta_k)$ of equivalence classes of integrable functions on $\mathfrak{a}$ with respect to $\delta_k$ and denote by $\nrm{\m}_{1,\delta_k}$ the associated norm.

\begin{theorem}\label{CherednikIsomorphism}
Let $C_k$ be the constant from Theorem \ref{HJ}.
Then the following generalized Riemann-Lebesgue lemma holds for $f \in L^1(\mathfrak{a},\delta_k)$:
\begin{enumerate}[itemsep=5pt, topsep=5pt]
\item[\rm (i)] $\sup\limits_{z \in \mathfrak{a}+iC(\rho(k))}\abs{\mathcal{H}_kf(z)} \le C_k\nrm{f}_{1,\delta_k}$.
\item[\rm (ii)] The Cherednik transform $\mathcal{H}_kf$ is continuous on $\mathfrak{a}+iC(\rho(k))$ with
\begin{equation}\label{Decay}
\lim\limits_{\substack{\abs{\lambda} \to \infty \\ \Im \, \lambda \in C(\rho(k))}} \mathcal{H}_kf(\lambda)=0.
\end{equation}
\item[\rm (iii)] If $R$ is crystallographic, then the interior of $C(\rho(k))\subseteq \mathfrak{a}$ is non-empty and the Cherednik transform $\mathcal{H}_kf$ is holomorphic in the interior of $\mathfrak{a}+iC(\rho(k))$.
\end{enumerate}
\end{theorem}

\begin{proof}
Part (i) is an immediate consequence of Theorem \ref{HJ}. Let $f \in L^1(\mathfrak{a},\delta_k)$. The continuity (and holomorphicity in the crystallographic case) of $\mathcal{H}_kf$ is obtained from Theorem \ref{HJ} and standard theorems on continuous and holomorphic parameter integrals. The decay \eqref{Decay} of $\mathcal{H}_kf$ follows for $f \in C_c^\infty(\mathfrak{a})$ from Theorem \ref{HJ}. For general $f \in L^1(\mathfrak{a},\delta_k)$ we obtain \eqref{Decay} by part (i), dominated convergence and the fact that $C_c^\infty(\mathfrak{a}) \subseteq L^1(\mathfrak{a},\delta_k)$ is a dense subspace.
\end{proof}

\section{Non-symmetric limit transition between $\mathrm{BC}_n$ and $\mathrm{A}_n$}
Consider inside $\R^n$ with the standard orthonormal basis $e_1,\ldots,e_n$, where $e_i$ is the vector with a $1$ at position $i$ and $0$ otherwise. We fix the root systems 
$$\mathrm{BC}_n\coloneqq \set{e_i,2e_i \mid 1\le i \le n} \cup \set{\pm(e_i\pm e_j) \mid 1\le i< j\le n} \subseteq \R^n$$
with Weyl group $W_B=\Z_2^n \ltimes \mathcal{S}_n$, where $\Z_2^n$ acts by sign changes, while $\mathcal{S}_n$ acts by permutation of the coordinates. We fix a non-negative multiplicity $\kappa=(k_1,k_2,k_3) \ge 0$, where $k_1$ is the value on $\pm e_i$, $k_2$ is the value on $\pm 2e_i$ and $k_3$ is the value on $\pm (e_i\pm e_j)$. Furthermore, we consider the positive roots
$$\mathrm{BC}_n^+=\set{e_i,2e_i \mid 1\le i \le n} \cup \set{e_i\pm e_j \mid 1\le i< j\le n}$$
with simple roots $\alpha_1,\ldots,\alpha_n$ defined by
$$\alpha_i= \begin{cases}
e_i-e_{i+1}, &1 \le i \le n-1, \\
e_n, & i=n.
\end{cases}$$
The weights and dominant weights are then given by
\begin{align*}
P^{\mathrm{BC}}&=\Z^n, \\
P^{\mathrm{BC}}_+&=\Lambda_+^n\coloneqq\set{\lambda \in \Z^n \mid \lambda_1\ge \ldots \ge \lambda_n \ge 0},
\end{align*}
respectively. The positive Weyl chamber is
$$C_+^{\mathrm{BC}}\coloneqq \set{\lambda \in \R^n \mid \lambda_1>\ldots > \lambda_n>0}$$
and the Weyl vector is
$$\rho^{\mathrm{BC}}(\kappa) = \frac{1}{2}\sum\limits_{i=1}^n (k_1+2k_2+2k_3(n-i))e_i.$$
The highest short root of $\mathrm{BC}_n$ is $\beta\coloneqq e_1$ and we write
\begin{align*}
s_0 &\coloneqq \beta+ s_\beta \; \, = x \mapsto (1-x_1,x_2,\ldots,x_n), \\
s_i &\coloneqq s_{e_i-e_{i+1}}= x \mapsto (x_1,\ldots,x_{i-1},x_{i+1},x_i,x_{i+2},\ldots,x_n), \text{ for } 1\le i< n, \\
s_n &\coloneqq s_{e_n} \quad \quad \, = x\mapsto (x_1,\ldots,x_{n-1},-x_n).
\end{align*}
We denote the space of trigonometric polynomials of type $\mathrm{BC}_n$ by $\mathcal{T}^{\mathrm{BC}}$ and the dual affine Weyl group by $W_B^{\vee,\mathrm{aff}}=\braket{s_0,\ldots,s_n}_{\text{group}}$. Then Corollary \ref{PolRec} leads to the following:

\begin{corollary}\label{InductiveHOP}
Let $(E_\mu^{\mathrm{BC}}(\kappa;\m))_{\mu \in P^{\mathrm{BC}}}$ be the non-symmetric Heckman-Opdam polynomials associated with $(\mathrm{BC}_n^+,\kappa)$. Then we have
\begin{enumerate}[itemsep=5pt, topsep=5pt]
\item[\rm (i)] The minuscle weights are $\mathcal{O}=\set{0}$.
\item[\rm (ii)] For $\mu \in P^{\mathrm{BC}}=\Z^n$ with $s_i\mu \neq \mu$ for some $0\le i \le n$, there exists $d_i=d_i(\kappa;\mu)$ with 
$$d_iE_{s_i\mu}^{\mathrm{BC}}(\kappa;\m)=(s_i+c_i(\kappa;\mu))E_\mu^{\mathrm{BC}}(\kappa;\m)$$ and the constant
$$c_i(\kappa;\mu)=\begin{dcases}
\frac{k_1}{1-2\widetilde{\mu}_1}, & i=0, \\
\tfrac{k_3}{\widetilde{\mu}_i-\widetilde{\mu}_{i+1}}, &1\le i< n, \\
\tfrac{k_1+2k_2}{2\widetilde{\mu}_{n}},  & i=n.
\end{dcases}$$
The eigenvalue vector is given by
\begin{align*}
\tilde{\mu}&=\mu + \frac{k_1+2k_2}{2}\sum\limits_{i=1}^n\epsilon(x_i)e_i + \frac{k_3}{2}\sum\limits_{1\le i<j\le n} \epsilon(x_i \pm x_j) (e_i\pm e_j).
\end{align*}
\end{enumerate}
\end{corollary}

\begin{theorem}\label{RecurrenceLimit}
Let $\trianglelefteq$ be the partial order on $P^{\mathrm{BC}}=\Z^n$ defined in Definition \ref{DefHOPol} and $\lambda \in \Z^n$. Then there exists a trigonometric polynomial
$$E_{\lambda}^{\mathrm{BC}}(\infty;k_3;\m)=e^\lambda+\sum\limits_{\mu \triangleleft \lambda} c_{\mu\lambda}(k_3)\,e^\lambda \in \mathcal{T}^{\mathrm{BC}}$$ such that for fixed $k_3 \ge 0$:
\begin{enumerate}[itemsep=5pt, topsep=5pt]
\item[\rm (i)]$\lim\limits_{\substack{k_1+k_2 \to \infty \\ k_1/k_2 \to \infty}}E_\lambda^{\mathrm{BC}}(\kappa;\m)=E_\lambda^{\mathrm{BC}}(\infty;k_3;\m)$
locally uniformly on $\C^n$, including $k_2=0$ with $k_1/k_2=\infty$.
\item[\rm (ii)] If $w=s_{i_1}\cdots s_{i_m} \in W_{B}^{\vee,\mathrm{aff}}$ is reduced and of minimal length with $w\lambda=0$, then
\begin{equation}\label{InftyRecursion}
E_{\lambda}^{\mathrm{BC}}(\infty;k_3;\m)= \begin{cases}
1, & \lambda=0, \\
(s_{i_m}+ c_m) \cdots  (s_{i_1}+c_1)\m 1, & \lambda \neq 0,
\end{cases}
\end{equation}
where $c_\ell=c_{i_\ell}(\infty;k_3;s_{i_{\ell-1}}\cdots s_{i_1}\lambda)\ge 0$ is defined via
$$c_\ell(\infty;k_3;x) =\lim\limits_{\substack{k_1+k_2 \to \infty \\ k_1/k_2 \to \infty}} c_\ell(\kappa;x)$$
with $c_\ell(\kappa;x)$ as in Corollary \ref{InductiveHOP}. 
\item[\rm (iii)] Let $\epsilon(t)=-1$ if $t \le 0$ and $\epsilon(t)=1$ if $t>0$. The limits $c_i(\infty;k_3;x)$ satisfy
$$c_i(\infty;k_3;x)=\begin{cases}
-\epsilon(x_1), &i=0, \\
\epsilon(x_n), & i=n.
\end{cases}$$
For $1\le i <n $ we have $c_i(\infty;k_3;x)=0$ if $\epsilon(x_i)\neq \epsilon(x_{i+1})$ and otherwise
$$c_i(\infty;k_3;x)=\tfrac{k_3}{x_i-x_{i+1}+\tfrac{k_3}{2}\left(\sum\limits_{j>i}\epsilon(x_i\pm x_j)+\sum\limits_{j<i}\delta(x_j, x_{i})-\sum\limits_{j>i+1}\epsilon(x_{i+1}\pm x_j)-\sum\limits_{j<i+1}\delta(x_j, x_{i+1})\right)},$$
with $\delta(x_j,x_i)=\epsilon(x_j+x_i)-\epsilon(x_j-x_i)$.
\end{enumerate}
\end{theorem}

\begin{proof}
In view of the Corollaries \ref{PolRec} and \ref{InductiveHOP} it suffices to compute the limit $c_j(\infty;k_3;x)$.
\begin{itemize}[itemsep=5pt, topsep=5pt]
\item As $\widetilde{x}_1=x_1+\tfrac{1}{2}\left((k_1+2k_2)\epsilon(x_1)+k_3\sum\limits_{i>1} \epsilon(x_1\pm x_i)\right)$ we obtain
$$\frac{1-2\widetilde{x}_1}{k_1}\xrightarrow[\substack{k_1+k_2 \to \infty \\ k_1/k_2 \to \infty}]{} -\epsilon(x_1).$$ 
Thus, by $\epsilon(x_1)\in\set{\pm 1}$, we conclude $c_1(\infty;k_3;x)=\lim\limits_{\substack{k_1+k_2,\,k_1/k_2 \to \infty}}c_1(\kappa;x)= -\epsilon(x_1).$
\item Similar computations show that
$$c_n(\kappa;x)=\frac{k_1+2k_2}{2\widetilde{x}_n}\xrightarrow[\substack{k_1+k_2 \to \infty \\ k_1/k_2 \to \infty}]{} \epsilon(x_n)=c_n(\infty;k_3;x).$$
\item If $x$ satisfies $\epsilon(x_i)\neq \epsilon(x_{i+1})$ for $1\le i< n$, i.e. $\epsilon(x_{i+1})=-\epsilon(x_i)$, then
$$c_i(\kappa;x)=\frac{k_3}{\widetilde{x}_i-\widetilde{x}_{i+1}}=\frac{k_3}{x_i-x_{i+1}+(k_1+2k_2)\epsilon(x_i) + d(k_3,\mu)} \xrightarrow[\substack{k_1+k_2 \to \infty \\ k_1/k_2 \to \infty}]{} 0=c_i(\infty;k_3;x),$$
where $d(k_3,x) \in \R$ is independent of $k_1,k_2$. If conversely $\epsilon(x_i)=\epsilon(x_{i+1})$, then
$$c_i(\kappa;x)=\frac{k_3}{\widetilde{x}_i-\widetilde{x}_{i+1}}=\frac{k_3}{x_i-x_{i+1}+d(k_3,x)}=c_i(\infty,k_3;x)$$
is independent of $k_1,k_2$ and $d(k_3,x)$ is given by
\begin{align*}
d(k_3,x)&=\frac{k_3}{2}\!\left(\sum\limits_{j>i}\epsilon(x_i\pm x_j)+\sum\limits_{j<i}\delta(x_j, x_{i})-\!\!\!\sum\limits_{j>i+1}\!\!\epsilon(x_{i+1}\pm x_j)-\!\!\!\sum\limits_{j<i+1}\delta(x_j, x_{i+1})\right).
\end{align*}
\end{itemize}
\end{proof}

\begin{corollary}\label{ArbitraryWeylAction}
For $\lambda \in \Z^n$ with $s_i\lambda \neq \lambda$, there exists $d_i\coloneqq d_i(k_3;\lambda)\in \R$ with
$$d_iE_{s_i\lambda}^{\mathrm{BC}}(\infty;k_3;\m)=(s_i+c_i)E_\lambda^{\mathrm{BC}}(\infty;k_3;\m),$$
with $c_i=c_i(k_3;\lambda)$.
\end{corollary}

\begin{proof}
From Theorem \ref{InductiveHOP} we have a constant $d_i(\kappa;\mu) \in \R$ with
$$d_i(\kappa;\lambda)E_{s_i\mu}^{\mathrm{BC}}(\kappa;\m)=(s_i+c_i(\kappa;\lambda))E_\lambda^{\mathrm{BC}}(\kappa;\m),$$
i.e. $d_i(\kappa;\mu)=\frac{(s_i+c_i(\kappa;\lambda))E_\lambda^{\mathrm{BC}}(\kappa;0)}{E_{s_i\lambda}^{\mathrm{BC}}(\kappa;0)}.$
Hence, the claim holds by Theorem \ref{RecurrenceLimit} with
$$d_i=\lim\limits_{\substack{k_1+k_2 \to \infty \\ k_1/k_2 \to \infty}} d_i(\kappa;\lambda) = \frac{(s_i+c_i)E_\lambda^{\mathrm{BC}}(\infty;k_3;0)}{E_{s_i\lambda}^{\mathrm{BC}}(\infty;k_3;0)}.$$
\end{proof}

We further recall the recurrence relations for non-symmetric Jack polynomials which can be found in \cite[Proposition 12.2.1, Proposition 12.2.3]{F10}, see also \cite{BR23} for some application in the context of Laplace transforms. We do not give a precise definition of the Jack polynomials here. Instead, we describe them by there recurrence relations.

\begin{proposition}\label{JackRec}
Let $(E^{\mathrm{Jack}}_\lambda(k;\m))_{\lambda \in \N_0^n}$ be the non-symmetric Jack polynomials of index $\alpha=\tfrac{1}{k}$. Then, the Jack polynomials can be constructed from $E^{\mathrm{Jack}}_0(k;\m)\equiv 1$ by the following operations:
\begin{enumerate}[itemsep=5pt, topsep=5pt]
\item[\rm (i)] $E^{\mathrm{Jack}}_{\Phi\lambda}(k;\m)=\Phi E^{\mathrm{Jack}}_\lambda(k;\m)$ with the Knop-Sahi raising operator $\Phi$ (cf. \cite{KS97}) defined on $\lambda \in \N_0^n$ and functions $f:\R^n \to \C$ by
\begin{align*}
\Phi\lambda &=(\lambda_2,\ldots,\lambda_n,\lambda_1+1), \\
\Phi f(x)&=x_nf(x_n,x_1,\ldots,x_{n-1}).
\end{align*}
\item[\rm (ii)] $E^{\mathrm{Jack}}_{s_i\lambda}(k;\m)=(s_i+\tfrac{k}{\overline{\lambda}_{i+1}-\overline{\lambda}_i})E^{\mathrm{Jack}}_\lambda(k;\m)$ for $\lambda_i<\lambda_{i+1}$ and $i=1,\ldots,n-1$ with
$$\overline{\lambda}_i=\lambda_i-k\#\set{j<i\mid \lambda_j \ge \lambda_i}-k\#\set{j>i \mid \lambda_j>\lambda_i}.$$
\end{enumerate}
\end{proposition}

\begin{remark}
Recall from Definition \ref{DefHOPol} the ordering $\trianglelefteq$ on $P^{\mathrm{BC}}=\Z^n$ defined by
$$\mu \trianglelefteq \lambda \quad \text{ iff } \quad \begin{cases}
\mu_+ \le \lambda_+, &  \lambda_+ \neq \mu_+, \\
\hspace{7pt}\lambda \le \mu, & \lambda_+=\mu_+,
\end{cases}$$
where the $\lambda_+$ is the unique element in $W_B.\lambda \cap C_+^{\mathrm{BC}}$ and $\le$ is the dominance order defined by
\begin{align*}\label{\mathrm{BC}Dominance}
\mu \le \lambda &\quad\text{ iff }\quad \lambda-\mu \in Q_+=\mathrm{span}_{\N_0}(\mathrm{BC}_n^+)=\set*{\left.a_1e_1+\sum\nolimits_{i=2}^{n}(a_i-a_{i+1})e_i \right| a_i \in \N}  \\
&\quad\text{ i.e. } \sum\limits_{i=1}^p \mu_i \le \sum\limits_{i=1}^p \lambda_i \text{ for all }p=1,\ldots,n.
\end{align*}
\end{remark}

From the previous remark one immediately verifies by short computations the following proposition:

\begin{proposition}\label{OrderingWeylAction}
Let $\mu,\lambda \in \Z^n$ with $\mu \triangleleft \lambda$.
\begin{enumerate}[itemsep=5pt, topsep=5pt]
\item[\rm (i)] If $\lambda_{i+1} < \lambda_i$, then $s_i\mu \neq \lambda$ for all $i=1,\ldots,n-1$.
\item[\rm (ii)] If $\lambda_n\ge 0$, then $s_n\mu \neq \lambda$.
\item[\rm (iii)] If $\lambda_1 \le 0$, then $s_0\mu \neq \lambda$.
\end{enumerate}
\end{proposition}


\begin{definition}
The subsequent operator on $\Z^n$ plays an important role in the following
$$\widetilde{\Phi}=s_n\cdots s_0 = \eta \mapsto (\eta_2,\ldots,\eta_n,\eta_1-1).$$
This operator is related to the raising operator $\Phi$ by the equation $\widetilde{\Phi}(-\eta)=-\Phi\eta$. Hence, the operator $\widetilde{\Phi}$ will be important to characterize the limits of the $\mathrm{BC}_n$ Heckman-Opdam polynomials as Jack polynomials.
\end{definition}

\begin{lemma}\label{LimitReccurenceSteps}
The limits $(E_\lambda^{\mathrm{BC}}(\infty;k_3;\m))_{\lambda \in \Z^n}$ satisfy the following recurrence relations.
\begin{enumerate}[itemsep=5pt, topsep=5pt]
\item[\rm (i)] If $\lambda \in -\N_0^n$, then 
$$E_{\widetilde{\Phi}\lambda}^{\mathrm{BC}}(\infty;k_3;\m) = (s_n+1)s_{n-1}\cdots s_1(s_0+1)E_{\lambda}^{\mathrm{BC}}(\infty;k_3;\m).$$
\item[\rm (ii)] If $\lambda=-\eta \in -\N_0^n$ with $\lambda_{i+1}<\lambda_i$, then 
$$E_{s_i\lambda}^{\mathrm{BC}}(\infty;k_3;\m)= (s_i+\tfrac{k_3}{\overline{\eta}_{i+1}-\overline{\eta}_i})E^{\mathrm{BC}}_\lambda(\infty;k_3;\m)$$ with $\overline{\eta}$ as in Proposition \ref{JackRec} with $k=k_3$.
\item[\rm (iii)] Let $1\le i_1<....<i_\ell \le n$ be the indices with $\lambda_{i_j}>0$ and define
$$\lambda^*=(s_0 \cdots s_{i_{\ell}-1}) \cdots (s_0 \cdots s_{i_1-1})\lambda \in -\N_0^n.$$
Then we have
$$E_{\lambda}^{\mathrm{BC}}(\infty;k_3;\m) = [s_{i_1-1}\cdots s_1 (s_0+1)]\cdots [s_{i_\ell-1}\cdots s_1(s_0+1)]E_{\lambda^*}^{\mathrm{BC}}(\infty;k_3;\m),$$
with convention
$s_j\cdots s_1(s_0+1)=1$ if $j=0$.
\end{enumerate}
\end{lemma}

\begin{proof}
The most important argument in this proof is based on the triangular form
$$E_\lambda^{\mathrm{BC}}(\infty;k_3;\m)=e^\lambda+\sum\limits_{\mu \triangleleft \lambda} c_{\mu\lambda}(k_3)e^\mu,$$
from Theorem \ref{InductiveHOP}.
\begin{enumerate}[itemsep=5pt, topsep=5pt]
\item[\rm (i)] We will divide this part into several steps.
\begin{itemize}[itemsep=5pt, topsep=5pt]
\item Since $\lambda_1 \le 0$, we can conclude that $\epsilon(\lambda_1)=-1$ and therefore due to Corollary \ref{ArbitraryWeylAction}
\begin{equation}\label{Eq1a}
d_0E_{s_0\lambda}^{\mathrm{BC}}(\infty,k_3;\m)= (s_0+1)E_{\lambda}^{\mathrm{BC}}(\infty,k_3;\m)
\end{equation}
for some constant $d_0$. According to Proposition \ref{OrderingWeylAction} (iii), we can compare the coefficients of $e^{s_0\lambda}$: the left hand side of \eqref{Eq1a} has coefficient $d_0$ and the right hand side $1$, i.e. $d_0=1$.
\item Consider $i=1,\ldots, n-1$. Then $\lambda^*\coloneqq s_{i-1}\dots s_0\lambda$ is given by
$$\lambda^*=(\lambda_2,\ldots,\lambda_{i-1},1-\lambda_1,\lambda_{i+1},\ldots,\lambda_n).$$
Thus, $\epsilon(\lambda_i^*)=1=-\epsilon(\lambda^*_{i+1})$. Proceeding by induction over $i=1,\ldots,n-1$ this leads by Corollary \ref{ArbitraryWeylAction} to 
\begin{equation}\label{Eq2a}
d_iE_{s_i\lambda^*}^{\mathrm{BC}}(\infty,k_3;\m)=s_iE_{\lambda^*}^{\mathrm{BC}}(\infty,k_3;\m)
\end{equation}
for some constant $d_i$. According to Proposition \ref{OrderingWeylAction} (part (i), since $\lambda^*_{i+1}\le 0 < \lambda_i^*$) we can compare the coefficients of $e^{s_i\lambda^*}$: the left hand side of \eqref{Eq2a} has coefficient $d_i$ and the right hand side $1$, i.e. $d_i=1$.
\item Let $\lambda^*=s_{n-1}\cdots s_0\lambda=s_n\tilde{Phi} \lambda$, then $\lambda^*_n=1-\lambda_1$, i.e. $\epsilon(\lambda^*_n)=1$. Hence, Corollary \ref{ArbitraryWeylAction} gives
\begin{equation}\label{Eq3a}
d_nE_{\widetilde{\Phi}\lambda}^{\mathrm{BC}}(\infty,k_3;\m)=(s_n+1)E_{\lambda^*}^{\mathrm{BC}}(\infty,k_3;\m)
\end{equation}
for some constant $d_n$. According to Proposition \ref{OrderingWeylAction} (ii),  we can compare the coefficients of $e^{s_n\lambda^*}$: the left hand side of \eqref{Eq3a} has coefficient $d_n$ and the right hand side $1$, i.e. $d_n=1$.
\end{itemize}
The assertion follows by combining these three steps.
\item[\rm (ii)]  Consider $\lambda=-\eta \in -\N_0^n$ with $\lambda_i>\lambda_{i+1}$, i.e. $\eta_i<\eta_{i+1}$. 
\begin{itemize}[itemsep=5pt, topsep=5pt]
\item We claim that
\begin{equation}\label{ReflectionRecurrence}
\begin{split}
\overline{\eta}_{i+1}-\overline{\eta}_i &=  \lambda_i-\lambda_{i+1}+\tfrac{k_3}{2}\Big(\sum\limits_{j>i}\epsilon(\lambda_i\pm\lambda_j)+\sum\limits_{j<i}\delta(\lambda_j,\lambda_{i}) \\
&\quad -\sum\limits_{j>i+1}\epsilon(\lambda_{i+1}\pm\lambda_j)-\sum\limits_{j<i+1}\delta(\lambda_j,\lambda_{i+1})\Big),
\end{split}
\end{equation}
with $\delta$ as in Corollary \ref{ArbitraryWeylAction}. According to $\lambda=-\eta \in \N_0^n$ and $\eta_i<\eta_{i+1}$, we have 
\begin{align*}
\sum\limits_{j>i}&\epsilon(\lambda_i\pm\lambda_j)+\sum\limits_{j<i}\delta(\lambda_j,\lambda_{i}) -\sum\limits_{j>i+1}\epsilon(\lambda_{i+1}\pm\lambda_j)-\sum\limits_{j<i+1}\delta(\lambda_j,\lambda_{i+1}) \\
&= -(n-i)+\#\set{j>i \mid \lambda_i>\lambda_j}-\#\set{j>i \mid \lambda_i \le \lambda_j} \\
&\quad -(i-1)+\#\set{j<i \mid \lambda_j\le \lambda_i}-\#\set{j<i\mid \lambda_j>\lambda_i} \\
&\quad +(n-(i+1))-\#\set{j>i+1 \mid \lambda_{i+1}>\lambda_j}-\#\set{j>i+1 \mid \lambda_{i+1} \le \lambda_j} \\
&\quad +i+\#\set{j<i+1 \mid \lambda_j\le \lambda_{i+1}}-\#\set{j<i+1\mid \lambda_j>\lambda_{i+1}}  \\
&=\#\set{j>i\mid \eta_j>\eta_i}+\#\set{j<i\mid \eta_j\ge \eta_i} \\
&\quad -\#\set{j>i+1 \mid \eta_j>\eta_{i+1}}-\#\set{j<i+1 \mid \eta_j\ge \eta_{i+1}} \\
&\quad -\#\set{j>i\mid \eta_j\le \eta_i}+\#\set{j>i+1\mid \eta_j \le \eta_{i+1}} \\
&\quad -\#\set{j<i\mid \eta_j<\eta_i}+\#\set{j<i+1\mid \eta_j<\eta_{i+1}} \\
&=2 \m \Big(\#\set{j>i\mid \eta_j>\eta_i}+\#\set{j<i\mid \eta_j\ge \eta_i} \\
&\quad -\#\set{j>i+1 \mid \eta_j>\eta_{i+1}}-\#\set{j<i+1 \mid \eta_j\ge \eta_{i+1}}\Big).
\end{align*}
From here it is immediate that equation \eqref{ReflectionRecurrence} is true.
\item By Corollary \ref{ArbitraryWeylAction} and \eqref{ReflectionRecurrence} there exists a constant $d_i$ with
\begin{equation}\label{Eq4}
d_iE_{s_i\lambda}^{\mathrm{BC}}(\infty,k_3;\m)=(s_i+\tfrac{k_3}{\overline{\eta}_{i+1}-\overline{\eta}_i})E_{\lambda}^{\mathrm{BC}}(\infty,k_3;\m).
\end{equation}
According to Proposition \ref{OrderingWeylAction} (ii), we can compare the coefficients of $e^{s_i\lambda}$: the left hand side of \eqref{Eq4} has coefficient $d_i$ and the right hand side $1$, i.e. $d_i=1$. 
\end{itemize}
\item[\rm (iii)] Assume that $\lambda \in \Z^n$ has exactly $\ell$ positive entries, namely for $1\le i_1 < \ldots < i_\ell \le n$ with $\lambda_{i_j}>0$. We define
$$\lambda'\coloneqq s_0\cdots s_{i_1-1}\lambda = (1-\lambda_{i_1},\lambda_2,\ldots,\lambda_{i_1-1},\lambda_{i_1+1},\ldots,\lambda_n).$$
In particular, $\lambda'_j \le 0$ for $1\le j < i_2$. By the proof of part (i) (the first two steps of the proof of part (i)) we have
$$E_{\lambda}^{\mathrm{BC}}(\infty;k_3;\m) = E_{s_{i_1-1}\cdots s_0\lambda'}^{\mathrm{BC}}(\infty;k_3;\m) = s_{i_1-1}\cdots s_1 (s_0+1)E_{\lambda'}^{\mathrm{BC}}(\infty;k_3;\m).$$
Since $\lambda'$ has precisely $\ell-1$ we can proceed by induction to obtain the claimed formula.
\end{enumerate}
\end{proof}

\begin{theorem}\label{BCJack}
For $x \in \R^n$ and $f:\R \to \R$ define $f(x)\coloneqq (f(x_1),\ldots,f(x_n)) \in \R^n$ componentwise. 
The limits $(E_\lambda^{\mathrm{BC}}(\infty;k_3;\m))_{\lambda \in \N_0^n}$ are explicitly given by
\begin{enumerate}[itemsep=5pt, topsep=5pt]
\item[\rm (i)] If $\eta \in \N_0^n$, then $E_{-\eta}^{\mathrm{BC}}(\infty;k_3;x) = 4^{\abs{\eta}}E_\eta^{\mathrm{Jack}}(k_3;\cosh^2(\tfrac{x}{2}))$. 
\item[\rm (ii)] If $\lambda \in \Z^n\setminus(- \N_0^n)$ and $1\le i_1<\ldots <i_\ell \le n$ are the indices with $\lambda_{i_j}>0$. Then for
$$\lambda^*=(s_0 \cdots s_{i_{\ell}-1}) \cdots (s_0 \cdots s_{i_1-1})\lambda = (1-\lambda_{i_\ell},\ldots, 1-\lambda_{i_1},\lambda') \in -\N_0^n,$$
where $\lambda' \in \N_0^{n-\ell}$ is the vector $\lambda$ with deleted entries $\lambda_{i_j}$, it holds
\begin{align*}
E_\lambda^{\mathrm{BC}}(\infty;k_3;\m) &= (e^{x_{i_1}}+1)\cdots(e^{x_{i_\ell}}+1)E_{\lambda^*}^{\mathrm{BC}}(\infty;k_3;x_{i_\ell},\ldots,x_{i_1},x') \\
&= 4^{\abs{\lambda}-\ell}(e^{x_{i_1}}+1)\cdots(e^{x_{i_\ell}}+1)E_{-\lambda^*}^{\mathrm{Jack}}(k_3;\cosh^2(\tfrac{x_{i_\ell}}{2},\ldots,\tfrac{x_{i_1}}{2},\tfrac{x'}{2})),
\end{align*}
where $x' \in \R^{n-\ell}$ is the vector with deleted entries $x_{i_j}$. Moreover, $E_\lambda^{\mathrm{BC}}$ is $\Z_2$-invariant in the variables $x_j$ with $j \neq i_1,\ldots,i_\ell$.  
\item[\rm (iii)] In the situation of part (ii) it holds
$$\frac{1}{2^n}\sum\limits_{\tau \in \Z^n} E_{\lambda}^{\mathrm{BC}}(\infty;k_3;\tau x) = 4^{\abs{\lambda}-\tfrac{\ell}{2}}\prod\limits_{j=1}^\ell \cosh^2(\tfrac{x_{i_j}}{2}) \m E_{-\lambda^*}^{\mathrm{Jack}}(k_3;\cosh^2(\tfrac{\sigma_\lambda x}{2})),$$
with $\sigma_\lambda = (s_1\cdots s_{i_\ell-1})\cdots (s_1\cdots s_{i_1-1})$, i.e. $\sigma_\lambda x= (x_{i_\ell},\ldots,x_{i_1},x')$. This formula also makes sense for $\lambda \in -\N_0^n$ with $\lambda^*=\lambda$ and $\sigma_\lambda=1$ by part (i).
\item[\rm (iv)] For $\lambda \in \Z^n$ it holds for $1\le i_1<\ldots<i_\ell\le n$ the indices with $\lambda_{i_j}>0$
$$\frac{1}{2^n}\sum\limits_{\tau \in \Z^n} E_{\lambda}^{\mathrm{BC}}(\infty;k_3;\tau x) = 4^{\abs{\lambda}-\tfrac{\ell}{2}}E_{-\lambda^{**}}^{\mathrm{Jack}}(k_3;\sigma_{\lambda}^*\cosh^2(\tfrac{x}{2})),$$
where $\lambda^{**}=(\lambda',-\lambda_{i_\ell},\ldots, -\lambda_{i_1})$ and $\sigma_\lambda^*=(s_{i_1}\cdots s_{n-1})\cdots (s_{i_\ell}\cdots s_{n-1})$, \\
i.e. $\sigma_\lambda^* x= (x',x_{i_\ell},\ldots,x_{i_1})$, where $x'$ and $\lambda'$ are choosen as before.
\end{enumerate}
\end{theorem}

\begin{proof}
\
\begin{enumerate}[itemsep=5pt, topsep=5pt]
\item[\rm (i)] By Proposition \ref{JackRec} the family $f_\eta(x) \coloneqq 4^{\abs{\eta}}E_\eta^{\mathrm{Jack}}(k_3;\cosh^2(\tfrac{x}{2})),\eta \in \N_0^n$ is uniquely determined by
\begin{align*}
f_0(x)&=1 \\
f_{\Phi\eta}(x) &=4\cosh^2(\tfrac{x_n}{2})f_\eta(x_n,x_1,\ldots,x_{n-1}) \\
f_{s_i\eta}(x) &= (s_i+\tfrac{k_3}{\overline{\eta}_{i+1}-\overline{\eta}_i})f_\eta(x), \text{ if }  \eta_i<\eta_{i+1}.
\end{align*}
Hence, we prove that $g_\eta\coloneqq E_{-\eta}^{\mathrm{BC}}(\infty;k_3;\m),\eta \in \N_0^n$ satisfy the same equations. 
For $\eta=0$ the equation is immediate. Assume that the equations hold for $\eta \in \N_0^n$, in particular $g_\eta$ is $\Z_2^n$ invariant.
\begin{itemize}[itemsep=5pt, topsep=5pt]
\item $s_0$ acts on $\mathcal{T}^{\mathrm{BC}}$ by $s_0f(x)=e^{x_1}f(-x_1,x_2,\ldots,x_n)$ and therefore by the induction hypothesis and Lemma \ref{LimitReccurenceSteps}
\begin{align*}
g_{\Phi\eta}(x) &= E_{-\Phi\eta}^{\mathrm{BC}}(\infty;k_3;x) = E_{\widetilde{\Phi}(-\eta)}^{\mathrm{BC}}(\infty;k_3;x) \\
&= (s_n+1)s_{n-1}\cdots s_1(s_0+1)E_{-\eta}^{\mathrm{BC}}(\infty;k_3;x) \\
&= (s_n+1)s_{n-1}\cdots s_1 \left(e^{x_1}g_\eta(-x_1,x_2,\ldots,x_n)+g_\eta(x_1,\ldots,x_n) \right) \\
&= (s_n+1)s_{n-1}\cdots s_1 (e^{x_1}+1)g_\eta(x) \\
&= (s_n+1)(e^{x_n}+1)g_\eta(x_n,x_1,\ldots,x_{n-1}) \\
&= (e^{-x_n}+1)g_\eta(-x_n,x_1,\ldots,x_{n-1}) + (e^{x_n}+1)g_\eta(x_n,x_1,\ldots,x_{n-1}) \\
&= 4\cosh^2(\tfrac{x_n}{2}) g_\eta(x_n,x_1,\ldots,x_{n-1}).
\end{align*}
\item If $\eta_i<\eta_{i+1}$, then $(-\eta)_{i+1}<(-\eta)_i$ and therefore by the induction hypothesis and Lemma \ref{LimitReccurenceSteps}
\begin{align*}
g_{s_i\eta}(x) &= E_{s_i(-\eta)}^{\mathrm{BC}}(\infty;k_3;\m)  = (s_i+\tfrac{k_3}{\overline{\eta}_{i+1}+\overline{\eta}_i})E_{-\eta}^{\mathrm{BC}}(\infty;k_3;x) \\
&=(s_i+\tfrac{k_3}{\overline{\eta}_{i+1}+\overline{\eta}_i})g_\eta(x).
\end{align*}
\end{itemize}
Thus we conclude $g_\eta=f_\eta$ for all $\eta \in \N_0^n$.
\item[\rm (ii)] This is immediate from Lemma \ref{LimitReccurenceSteps}.
\item[\rm (iii)] Since $E_{\lambda^*}^{\mathrm{BC}}(\infty;k_3;\m)$ is $\Z_2^n$-invariant, this formula is immediate from parts (i)+(ii) together with $\abs{\lambda^*}=\abs{\lambda}-\ell$ and
\begin{align*}
\tfrac{1}{2^n}\sum\limits_{\tau \in \Z_2^n} (e^{(\tau x)_{i_1}}+1)\cdots(e^{(\tau x)_{i_\ell}}+1) &= \tfrac{1}{2^\ell} (e^{x_{i_1}}+e^{-x_{i_1}}+2) \cdots (e^{x_{i_n}}+e^{-x_{i_n}}+2) \\
&= 2^{\ell} \cosh^2(\tfrac{x_{i_1}}{2})\cdots \cosh^2(\tfrac{x_{i_\ell}}{2}).
\end{align*}
\item[\rm (iv)] This can be inductively constructed from part (iii), since the Jack polynomials satisfy
\begin{align*}
&\quad \quad E_{\Phi\eta}^{\mathrm{Jack}}(k_3;y)=\Phi E_\eta^{\mathrm{Jack}}(k_3;y),\\
&\text{ i.e. } E_{(\eta_2,\ldots,\eta_n,\eta_1+1)}^{\mathrm{Jack}}(k_3;y) = y_nE_\eta^{\mathrm{Jack}}(k_3;y_n,y_1,\ldots,y_{n-1}).
\end{align*}
\end{enumerate}
\end{proof}

The following result was already proven in \cite[Theorem 4.2]{RKV13}, but we give an independent proof, based on the results for the non-symmetric setting.

\begin{theorem}
Let $(P_\lambda^{\mathrm{Jack}}(k_3;\m))_{\lambda \in \Lambda_+^n}$ be the symmetric Jack polynomials of index $\alpha=\tfrac{1}{k_3}$, cf. \cite[Section 12.6]{F10}. Then for fixed $k_3\ge 0$ and $\kappa=(k_1,k_2,k_3)\ge 0$ the symmetric Heckman-Opdam polynomials $(P^{\mathrm{BC}}_{\lambda}(\kappa;\m))_{\lambda \in \Lambda_+^n}$ satisfy the limit transition
$$P^{\mathrm{BC}}_{\lambda}(\kappa;z) \xrightarrow[\substack{k_1+k_2 \to \infty \\ k_1/k_2 \to \infty}]{} 4^{\abs{\lambda}}P_\lambda^{\mathrm{Jack}}(k_3;\cosh^2(\tfrac{z}{2})),$$
locally uniformly in $z \in \C^n$. 
\end{theorem}

\begin{proof}
Consider for $\underline{1}=(1,\ldots,1)\in\R^n$ the renormalizations
$$\widetilde{P}_\lambda^{\mathrm{BC}}(\kappa;\m)\coloneqq\frac{P_\lambda^{\mathrm{BC}}(\kappa;\m)}{P_\lambda^{\mathrm{BC}}(\kappa;0)}, \quad \widetilde{P}_\lambda^{\mathrm{Jack}}(k_3;\m)\coloneqq\frac{P_\lambda^{\mathrm{Jack}}(k_3;\m)}{P_\lambda^{\mathrm{Jack}}(k_3;\underline{1})}.$$
\
\begin{enumerate}[itemsep=5pt, topsep=5pt]
\item[\rm (i)] For the type $\mathrm{BC}$ Cherednik kernel $G^{\mathrm{BC}}_\kappa$ and hypergeometric Function $F^{\mathrm{BC}}_\kappa$ one has for $\lambda \in P^{\mathrm{BC}}_+=\Lambda_+^n$ that $\widetilde{-\lambda}=-\lambda-\rho^{\mathrm{BC}}(\kappa)$ and therefore
$$G^{\mathrm{BC}}_\kappa(-\lambda-\rho^{\mathrm{BC}}(\kappa),\m)=\frac{E_{-\lambda}^{\mathrm{BC}}(\kappa;\m)}{E_{-\lambda}^{\mathrm{BC}}(\kappa,0)}.$$
The longest element in $W_B$ is $w_0=-\mathrm{id}$ and therefore \eqref{HypGeoHOPoly} leads to
\begin{align*}
\frac{P_\lambda^{\mathrm{BC}}(\kappa;\m)}{P_\lambda^{\mathrm{BC}}(\kappa;0)}=\frac{1}{2^nn!}\sum\limits_{w \in W_B} \frac{E_{-\lambda}^{\mathrm{BC}}(\kappa;w\m)}{E_{-\lambda}^{\mathrm{BC}}(\kappa,0)}.
\end{align*}
The Jack polynomials also satisfy by \cite[Formula 12.100]{F10}
$$\frac{P_\lambda^{\mathrm{Jack}}(k_3;\m)}{P_\lambda^{\mathrm{Jack}}(k_3;\underline{1})} = \frac{1}{n!}\sum\limits_{\sigma \in \mathcal{S}_n}\frac{E_\lambda^{\mathrm{Jack}}(k_3;\sigma\m)}{E_\lambda^{\mathrm{Jack}}(k_3;\underline{1})}.$$
Since the map $\R^n \to \R^n,$ $x \mapsto \cosh^2(\tfrac{x}{2})$ is $W_B$-equivariant we obtain from Theorem \ref{BCJack} that locally uniformly
$$\widetilde{P}_\lambda^{\mathrm{BC}}(\kappa;x) = \frac{1}{2^nn!}\sum\limits_{w \in W_B} \frac{E_{-\lambda}^{\mathrm{BC}}(\kappa;wx)}{E_{-\lambda}^{\mathrm{BC}}(\kappa,0)} \xrightarrow[\substack{k_1+k_2 \to \infty \\ k_1/k_2 \to \infty}]{} 4^{\abs{\lambda}} \widetilde{P}_\lambda^{\mathrm{Jack}}(k_3;\cosh^2(\tfrac{x}{2})).$$
\item[\rm(ii)] To prove the limit transition without the renormalization, it suffices to check the limit
$$P_\lambda^{\mathrm{BC}}(\kappa;0) \xrightarrow[\substack{k_1+k_2 \to \infty \\ k_1/k_2 \to \infty}]{} 4^{\abs{\lambda}}P_\lambda^{\mathrm{Jack}}(k_3;\underline{1}).$$
Owing to Heckman, the value $P_\lambda^{\mathrm{BC}}(\kappa;0)$ can be expressed in terms of a generalized Harish-Chandra $c$-function, namely
$$P_\lambda^{\mathrm{BC}}(\kappa;0)=\prod\limits_{\alpha \in \mathrm{BC}_n^+} \frac{\Gamma\big(\braket{\rho^{\mathrm{BC}}(\kappa),\alpha^\vee}+\tfrac{k_{\alpha/2}}{2}\big) \Gamma\big(\braket{\lambda+\rho^{\mathrm{BC}}(\kappa),\alpha^\vee}+\tfrac{k_{\alpha/2}+2k_\alpha}{2}\big)}{\Gamma\big(\braket{\rho^{\mathrm{BC}}(\kappa),\alpha^\vee}+\tfrac{k_{\alpha/2}+2k_\alpha}{2}\big)\Gamma\big(\braket{\lambda+\rho^{\mathrm{BC}}(\kappa),\alpha^\vee}+\tfrac{k_{\alpha/2}}{2}\big)},$$
see for instance \cite{HO87, HS94}. Denote by $d_\alpha(k_1,k_2,k_3)$ the quotient behind the product sign for fixed $\alpha \in \mathrm{BC}_n^+$. The asymptotic equality 
$$\frac{\Gamma(z+w)}{\Gamma(z)} \approx z^w \text{ for } z \to \infty,\, \Re(z)>0$$
immediately leads to
$$d_\alpha(k_1,k_2,k_3) \xrightarrow[\substack{k_1+k_2 \to \infty \\ k_1/k_2 \to \infty}]{} \begin{cases}
4^{\lambda_i},& \quad \alpha=e_i, \\
1,& \quad \alpha=2e_i, \, e_i+e_j, \\
\frac{\Gamma((j-i)k_3)\Gamma(\lambda_i-\lambda_k+k_3+(j-i)k_3)}{\Gamma(\lambda_i-\lambda_k+(j-i)k_3)\Gamma(k_3+(j-i)k_3)},& \quad \alpha=e_i-e_j.
\end{cases}$$
Therefore, together with the Pochhammer symbol $(z)_\alpha = \tfrac{\Gamma(z+\alpha)}{\Gamma(z)}$ we have proven that
\begin{align*}
P^{\mathrm{BC}}_\lambda(\kappa,0) \xrightarrow[\substack{k_1+k_2 \to \infty \\ k_1/k_2 \to \infty}]{} & 4^{\abs{\lambda}}\prod\limits_{i<j} \frac{\Gamma((j-i)k_3)\Gamma(\lambda_i-\lambda_j+k_3+(j-i)k_3)}{\Gamma(\lambda_i-\lambda_k+(j-i)k_3)\Gamma(k_3+(j-i)k_3)} \\
&= 4^{\abs{\lambda}} \prod\limits_{i<j} (\lambda_i-\lambda_j+(j-i)k_3)_{k_3} \m \prod\limits_{1\le j\le n} \frac{\Gamma(k_3)}{\Gamma(jk_3)} \\
&=4^{\abs{\lambda}}P_\lambda^{\mathrm{Jack}}(k_3;\underline{1}),
\end{align*}
where the last equality can be found for instance in \cite[Formula (6.4)]{OO98}
\end{enumerate}
\end{proof}

\section{Limit transition of the Cherednik kernels}
Let $G_\kappa^{\mathrm{BC}}$ be the Cherednik kernel associated with $(\mathrm{BC}_n^+,\kappa)$ with $\kappa=(k_1,k_2,k_3)\ge 0$ as before. Furthermore, consider the type $\mathrm{A}$ root system inside $\R^n$ with positive roots
\begin{align*}
\mathrm{A}_{n-1}&=\set{\pm (e_i-e_j)\mid 1\le i<j\le n} \subseteq \R^n, \\
\mathrm{A}_{n-1}^+&= \set{e_i-e_j\mid 1\le i<j\le n}.
\end{align*}
Let $G_{k_3}^{\mathrm{A}}$ be the Cherednik kernel associated with $(\mathrm{A}_{n-1}^+,k_3)$ on $\R^n$ as introduced in Section \ref{IntegralRoots} and consider rational version on $G_{k_3}^{\mathrm{A}}$ namely
$$\mathcal{G}_{k_3}^{\mathrm{A}}(\lambda,x)=G_{k_3}^{\text{A}}(\lambda,\log x), \text{ for all } x\in (0,\infty)^n,$$
where $\log$ is the inverse of $\exp:\R^n \to \R_+^n$. This rational version of the Cherednik kernel was already studied in \cite{BR23} and is a holomorphic extension of the Jack polynomial $(E_\lambda^{\mathrm{Jack}}(k;\m))_{\lambda \in \N_0^n}$ in the variable $\lambda$. Indeed, for $\lambda \in \N_0^n$ we have by \cite[Theorem 5.1, Formula (3.2)]{BR23} that
$$\mathcal{G}_{k_3}^{\mathrm{A}}(\overline{\lambda}+ \tfrac{k_3}{2}(n-1,\ldots,n-1), x)=\frac{E_\lambda^{\mathrm{Jack}}(k_3,x)}{E_\lambda^{\mathrm{Jack}}(k_3;1,\ldots,1)}.$$
with $\overline{\lambda}=(\overline{\lambda}_1,\ldots,\overline{\lambda}_n)$ from Proposition \ref{JackRec}: If in addition $\lambda \in \Lambda_+^n$ holds, then 
$$\mathcal{G}_{k_3}^{\mathrm{A}}(\lambda-\rho^{\mathrm{A}}(k_3), x)=\frac{E_\lambda^{\mathrm{Jack}}(k_3;x)}{E_\lambda^{\mathrm{Jack}}(k_3;1,\ldots,1)}.$$
We are know able to prove the subsequent theorem, which generalizes the limit transition between the type $\mathrm{BC}$ hypergeometric function and the type $A$ hypergeometric function from \cite[Theorem 5.1]{RKV13}. 

\begin{theorem}\label{CherednikBCA}
For all $x \in \R^n$ and $\lambda \in \C^n$ 
\begin{align*}
G_{\kappa}^{\mathrm{BC}}(-\lambda-\rho^{\mathrm{BC}}(\kappa),x) \xrightarrow[\substack{k_1+k_2 \to \infty \\ k_1/k_2 \to \infty}]{} &\mathcal{G}_{k_3}^{\mathrm{A}}(\lambda-\rho^{A}(k_3),\cosh^2(\tfrac{x}{2})), \\
&=G^{\mathrm{A}}_{k_3}(\lambda-\rho^{\mathrm{A}}(k_3), \log(\cosh^2(\tfrac{x}{2}))), \\
&=\prod\limits_{1\le i \le n} (\cosh^2(\tfrac{x_i}{2}))^{\tfrac{\braket{\lambda,\underline{1}}}{2}} \m G^{\mathrm{A}}_{k_3}(\lambda-\rho^{\mathrm{A}}(k_3), \pi(\log(\cosh^2(\tfrac{x}{2})))),
\end{align*}
locally uniformly in $\lambda$, where $\pi:\R^n \to \R_0^n=\mathrm{span}_\R \mathrm{A}_{n-1}$ is the orthogonal projection.
\end{theorem}

\begin{proof}
The proof is exactly the same as in \cite[Theorem 5.1]{RKV13} and uses an analytic continuation argument. One just have to replace in \cite[Theorem 5.1]{RKV13} the hypergeometric functions by the Cherednik kernel and the associated symmetric polynomials by their non-symmetric analogs.  \\
The last two equality in the limit formula in Theorem \ref{CherednikBCA} are a consequence from Theorem \ref{ProdDecomp}.
\end{proof}

\begin{remark}
As a corollary we obtain the results from \cite[Theorem 5.1]{RKV13} for the hypergeometric function. 
Let $F^{\text{A}}_{k_3}$ be the type $\mathrm{A}$ hypergeometric function on $\R^n$ associated with $(\mathrm{A}_{n-1}^+,k_{3})$ and $\mathcal{F}_{k_3}^{\mathrm{A}}$ the rational hypergeometric function on $(0,\infty)^n$, i.e.
$$\mathcal{F}_k^{\mathrm{A}}(\lambda,x)=F_k^{\mathrm{A}}(\lambda,\log x) = \frac{1}{n!}\sum\limits_{\sigma \in \mathcal{S}_n} \mathcal{G}_k(\lambda,\sigma x).$$ Averaging in Theorem \ref{CherednikBCA} gives 
\begin{align*}
F^{\mathrm{BC}}_\kappa(\lambda,x) \xrightarrow[\substack{k_1+k_2 \to \infty \\ k_1/k_2 \to \infty}]{} & \mathcal{F}_{k_3}^{\mathrm{A}}(\lambda-\rho^{\mathrm{A}}(k_3),\cosh^2(\tfrac{x}{2})) \\
&= F^{\mathrm{A}}_{k_3}(\lambda-\rho^{\mathrm{A}}(k_3), \log(\cosh^2(\tfrac{x}{2}))), \\
&=\prod\limits_{1\le i \le n} (\cosh^2(\tfrac{x_i}{2}))^{\tfrac{\braket{\lambda,\underline{1}}}{2}} \m F^{\mathrm{A}}_{k_3}(\pi(\lambda)-\rho^{\mathrm{A}}(k_3), \pi(\log(\cosh^2(\tfrac{x}{2})))),
\end{align*}
where the last two equality holds by Theorem \ref{ProdDecomp}. \\
The different sign in the $\rho$-shift, compared to \cite[Theorem 5.1]{RKV13} comes from the different choice of positive subsystems in $\mathrm{A}_{n-1}$, i.e. $-\rho^{\mathrm{A}}(k_3)$ here is precisely the term $+\rho^{\mathrm{A}}(k_3)$ in \cite{RKV13}. 
\end{remark}

\begin{remark}
A limit transition for the $\mathrm{B}_n$ root system can be easily obtained from the type $\mathrm{BC}_n$ case. Choose for the root system $\mathrm{B}_n=\mathrm{BC}_n\setminus \set{2e_i\mid i=1,\ldots,n}\subseteq \R^n$ with positive subsystem
$$\mathrm{B}_n^+=\mathrm{BC}_n^+\cap \mathrm{B}_n = \set{e_i\mid 1\le i \le n} \cup \set{e_i\pm e_j \mid 1\le i< j\le n}.$$
Then, both root system have the same Weyl group, simple roots, weight lattice and dominant weights. Denote by $\kappa'=(k_1,k_3)\ge 0$ a multiplicity function on ${\mathrm{B}}_n$, where $k_1$ is the value on $e_i$ and $k_3$ is the value on $e_i\pm e_j$. Then the Cherednik operators $D_\xi^{\mathrm{B}}(\kappa')$ and $D_\xi^{\mathrm{BC}}(\kappa)$ for the given positive subsystems, $\xi \in \R^n$, are related by
$$D_\xi^{\mathrm{B}}(k_1,k_3)=D_\xi^{\mathrm{BC}}(k_1,0,k_3),$$
and also $\rho^{\mathrm{B}}(k_1,k_3)=\rho^{\mathrm{BC}}(k_1,0,k_3)$ .
Hence, the non-symmetric and symmetric Heckman-Opdam polynomials, as well as the Cherednik kernels and hypergeometric functions are related by
\begin{align*}
E^{\text{B}}_\lambda(k_1,k_3;z)&=E^{\mathrm{BC}}_\lambda(k_1,0,k_3;z), \quad \lambda \in \Z^n, \, k_1,k_3\ge 0, \, z\in \C^n; \\
P^{\text{B}}_\lambda(k_1,k_3;z)&=P^{\mathrm{BC}}_\lambda(k_1,0,k_3;z), \quad \lambda \in \Lambda_+^n, \, k_1,k_3\ge 0, \, z\in \C^n; \\
G^{\text{B}}_{(k_1,k_3)}(\lambda;x)&=G^{\mathrm{BC}}_{(k_1,0,k_3)}(\lambda,x), \quad \lambda \in \C^n, \, k_1,k_3\ge 0, \, x \in \R^n; \\
F^{\text{B}}_{(k_1,k_3)}(\lambda;x)&=F^{\mathrm{BC}}_{(k_1,0,k_3)}(\lambda,x), \quad \lambda \in \C^n, \, k_1,k_3\ge 0, \, x \in \R^n.
\end{align*}
Hence, all proven limit transitions are also correct for the root system $\mathrm{B}_n$ instead of $\mathrm{BC}_n$ in the upper sense, one has only to consider $k_2=0$.
\end{remark}

\section*{Acknowledgment}

\end{document}